\documentclass[a4paper, 10pt]{amsart}
\usepackage{amsfonts,mathrsfs,amsmath,amscd,amssymb}
\usepackage[pdftex]{graphicx}
\usepackage{hyperref}
\usepackage{enumitem}
\usepackage{color}
\newtheorem{theorem}{Theorem}[section]
\usepackage[english]{babel}
\usepackage[all,cmtip]{xy}

\newtheorem{proposition}[theorem]{\sf Proposition}
\newtheorem{lemma}[theorem]{\sf Lemma}
\newtheorem{definition}[theorem]{\sf Definition}
\newtheorem{corollary}[theorem]{\sf Corollary}
\newtheorem{remark}[theorem]{\sf Remark}

\def\C{\mathbb C}
\def \Z{\mathbb Z}
\def \Q{\mathbb Q}
\def \R{\mathbb R}

\def \O{\mathcal O}
\def \ra{\rightarrow}
\def \vp{\varphi}

\def \F{\mathcal F}
\def \E{\mathcal E}
\DeclareMathOperator{\Spec}{Spec}

\let\olddefinition\definition
\renewcommand{\definition}{\olddefinition\normalfont}

\let\oldremark\remark
\renewcommand{\remark}{\oldremark\normalfont}

\begin{document}


\author{Tanya Kaushal Srivastava}
\title{On Derived Equivalences of K3 Surfaces in Positive Characteristic}
\address{Department of Mathematics and Computer Science, Freie Universit\"at, Arnimallee 3, Berlin 14195}
\email{tks.rket@gmail.com}
\date{\today}

\begin{abstract}
For an ordinary K3 surface over an algebraically closed field of positive characteristic we show that every automorphism lifts to characteristic zero. Moreover, we show that the Fourier-Mukai partners of an ordinary K3 surface are in one-to-one correspondence with the Fourier-Mukai partners of the geometric generic fiber of its canonical lift. We also prove that the explicit counting formula for Fourier-Mukai partners of the K3 surfaces with Picard rank two and with discriminant equal to minus of a prime number, in terms of the class number of the prime, holds over a field of positive characteristic as well. We show that the image of the derived autoequivalence group of a K3 surface of finite height in the group of isometries of its crystalline cohomology has index at least two. Moreover, we provide an upper bound on the kernel of this natural cohomological descent map. 

Further, we give an extended remark in the appendix on the possibility of an F-crystal structure on the crystalline cohomology of a K3 surface over an algebraically closed field of positive characteristic and show that the naive F-crystal structure fails in being compatible with inner product. 
\end{abstract}



\maketitle

\tableofcontents

\section*{Acknowledgement}
 
The results contained in this article are a part of PhD thesis  written under the supervision of Prof. Dr. H\'el\`ene Esnault. I owe her special gratitude for guidance, support, continuous encouragement and inspiration she always provided me. I thank Berlin Mathematical School for the PhD fellowship.  I am greatly thankful to  Michael Groechenig, Vasudevan Srinivas, Fran\c cois Charles, Christian Liedtke, Daniel Huybrechts,  Martin Olsson, Max Lieblich, Lenny Taelman and Sofia Tirabassi  for many mathematical discussions and suggestions.


\section{Introduction}

The derived category of coherent sheaves on a smooth projective variety was first studied as a geometrical invariant by Mukai in the early 1980's. In case the smooth projective variety has an ample canonical or anti-canonical bundle, Bondal-Orlov \cite{BOderived} proved that, if two such varieties have equivalent bounded derived categories of coherent sheaves, then they are isomorphic. However, in general this is not true. The bounded derived  category of coherent sheaves is not an isomorphism invariant. Mukai \cite{MukaiAV} showed that for an Abelian variety over $\C$, its dual has equivalent bounded derived category. Moreover, in many cases it can be shown that the dual of an Abelian variety is not birational to it, which implies that derived categories are not even birational invariants, see \cite{HuyFM} Chapter 9. Similarly, Mukai showed in \cite{Mu} that for K3 surfaces over $\C$, there are non-isomorphic K3 surfaces with equivalent derived categories. This led to the natural question of classifying all derived equivalent varieties.\\

For K3 surfaces, the case of interest to us, this was completed over $\C$ in late 1990's by Mukai and Orlov (\cite{Mu} Theorem 1.4, \cite{Orlov1} Theorem 1.5) using Hodge theory along with the Global Torelli Theorem (see \cite{Barth} VIII Corollary 11.2, \cite{HuyLect} Theorem 7.5.3). As a consequence, it was shown that there are only finitely many non-isomorphic K3 surfaces with equivalent bounded derived categories (see Proposition \ref{finitechar0}) and a counting formula was also proved by Hosono \'et.al in \cite{HLOY}.  On the other hand, for K3 surfaces over a field of positive characteristic, a partial answer to the classification question was first given by Lieblich-Olsson \cite{LO} (see Theorem \ref{LOmainthm}) in early 2010's. They showed that there are only finitely many non-isomorphic K3 surfaces with equivalent bounded derived categories. We remark here that due to unavailability of a positive characteristic version of the global Torelli Theorem for K3 surfaces of finite height, it is currently not feasible to give a complete cohomological description of derived equivalent K3 surfaces. However, a description in terms of moduli spaces was given by Lieblich-Olsson. We also point out here that the proofs of these results go via lifting to characteristic zero and thus use the Hodge theoretic description given by Mukai and Orlov. Furthermore, Lieblich-Olsson \cite{LO2} also proved the derived version of the Torelli theorem using the Crystalline Torelli theorem for supersingular K3 surfaces. \\

Meanwhile in 1990's another school of thought inspired by string theory in physics led Kontsevich \cite{Ko} to propose the homological mirror symmetry conjecture which states that the bounded derived category $D^b(X)$ of coherent sheaves of a projective variety $X$ is equivalent (as a triangulated category) to the bounded derived category $D^bFuk(\check{X},\beta)$ of the Fukaya category $Fuk(\check{X},\beta)$ of a mirror $\check{X}$ with its symplectic structure $\beta$. Moreover, the symplectic automorphisms of $\check{X}$ induce derived autoequivalences of $D^b(X)$. This provided a natural motivation for the study of the derived autoequivalence group. \\

For K3 surfaces $X$ over $\C$, the structure of the group of derived autoequivalences was analyzed by Ploog in \cite{Ploog}, Hosono et al. in \cite{Hosono} and Huybrechts, et al. in \cite{Huy}. They showed that the image of $Aut(D^b(X))$ under the homomorphism
$$
Aut(D^b(X)) \ra O_{Hodge}(\widetilde{H(X, \Z)}),
$$
where $O_{Hodge}(\widetilde{H(X, \Z)})$ is the group of Hodge isometries of the Mukai lattice of $X$, has index 2. However, the kernel of this map has a description only in the special case when the Picard rank of $X$ is 1, given by \cite{BB}. \\

In this article, we study the above two questions in more details for the case of K3 surfaces over an algebraically closed field of positive characteristic. In Section \ref{PrelimsK3} we recall the notion of height of a K3 surface over a field of positive characteristic, the results on lifting K3 surfaces from characteristic $p$ to characteristic $0$, the moduli spaces of stable sheaves on a K3 surface and derived equivalences on K3 surfaces. We end this section by proving that height of a K3 surface remains invariant under derived equivalences (Lemma \ref{derivedinvheight}).  In Section \ref{Derivedautoeqncharp}, we address the question on the group of derived autoequivalences for K3 surfaces of finite height.  We show that the image of the derived autoequivalence group of a K3 surface of finite height in the group of isometries of its crystalline cohomology has index at least two (Theorem \ref{orientation preserving}). Moreover, we provide an upper bound on the kernel of this natural cohomological descent map (Proposition \ref{boundonKernel}).  In Section \ref{CountingFMP}, we count the number of Fourier-Mukai partner for an ordinary K3 surface (Theorem \ref{ordinarycount}) along with showing that the automorphism group lifts to characteristic 0 (Theorem \ref{liftingauto}).  We also prove that the explicit counting formula for Fourier-Mukai partners of the K3 surfaces with Picard rank two and with discriminant equal to minus of a prime number, in terms of the class number of the prime, holds over a field of positive characteristic as well (Theorem \ref{Classno}). In Appendix  \ref{F-crystalchap}, we define an F-crystal structure and show that this integral structure is preserved by derived equivalences but its compatibility with intersection pairing fails.


\subsection{Conventions and Notations} \label{convention} 

For a field $k$ of positive characteristic $p$, $W(k)$ will be its ring of Witt vectors. For any cohomology theory $H^*_{...}(...)$, we will denote the dimension of the cohomology groups $H^i_{\ldots}(\ldots)$ as $h^i_{\ldots}(\ldots)$.  We will implicitly assume that the cardinality of $ K := Frac(W(k))$ and its algebraic closure $\bar{K}$ are not bigger than that of $\C$, this will allow us to choose an embedding $\bar{K} \hookrightarrow \C$ which we will use in our arguments to transfer results from characteristic $0$ to characteristic $p$. See also Remarks \ref{char0} and \ref{finiteChar0}.

\section{Preliminaries on K3 Surfaces and Derived Equivalences} \label{PrelimsK3}

We recall the notion height of a K3 surface ovewr a field of positive characteristic through its F-crystal, which gives a subclass of K3 surfaces with finite height or infinite height called supersingular K3 surfaces. For an introduction to Brauer group of K3 surfaces and the definition of height via the Brauer groups see \cite{HuyLect} and \cite{Liedtke}. Both definitions turn out to be equivalent (see, for example, Prop. 6.17 \cite{Liedtke}). 

Let $k$ be an  algebraically closed field of positive characteristic, $W(k)$ its ring of Witt vectors and $Frob_W$ the Frobenius morphism of $W(k)$ induced by the Frobenius automorphism of $k$. Note that $Frob_W$ is a ring homomorphism and induces an automorphism of the fraction fields $K := Frac (W(k))$, denoted as $Frob_K$. We begin by recalling the notion of F-isocrystal and F-crystals which we will use later to stratify the moduli of K3 surfaces. 

\begin{definition}[F-(iso)crystal] \label{F-crystal}
An \textbf{F-crystal} $(M, \phi_M)$ over $k$ is a free $W$-module $M$ of finite rank together with an injective $Frob_W$-linear map $\phi_M: M \ra M$, that is, $\phi_M$ is additive, injective and satisfies 
$$
\phi_M (r \cdot m) = Frob_W(r) \cdot \phi_M(m) \ \text{for all} \ r \in W(k), m \in M.
$$
An \textbf{F-isocrystal} $(V, \phi_V)$ is a finite dimensional $K$-vector space $V$ together with an injective $Frob_K$-linear map $\phi_V: V \ra V$. 

A \textbf{morphism $u:(M, \phi_M) \ra (N, \phi_N)$ of F-crystals} (resp. \textbf{F-isocrystals}) is a $W(k)$-linear (resp. $K$-linear) map $M \ra N$ such that $\phi_N \circ u = u \circ \phi_M$. An \textbf{isogeny} of F-crystals is a morphism $u: (M, \phi_M) \ra (N, \phi_N)$ of F-crystals, such that the induced map $ u \otimes Id_K: M \otimes_{W(k)} K \ra N \otimes_{W(k)} K$ is an isomorphism of F-isocrystals.
\end{definition}

\textbf{Examples:}
\begin{enumerate}
\item The trivial crystal: $(W, Frob_W)$.
\item  This is the case which will be of most interest to us:\\
Let $X$ be a smooth and proper variety over $k$. For any $n$, take the free $W(k)$ module $M$ to be $H^n := H^n_{crys}(X/W(k))/torsion$ and $\phi_M$ to be the Frobenius $F^*$. The Poincar\'e duality induces a perfect pairing 
\begin{equation*}
\langle -, - \rangle: H^n \times H^{2dim(X)-n} \ra H^{2dim(X)} \cong W 
\end{equation*}
which satisfies the following compatibility with Frobenius 
\begin{equation*}
\langle F^*(x), F^*(y) \rangle = p^{dim(X)} Frob_W(\langle x, y \rangle),
\end{equation*}
where $x \in H^n$ and $y \in H^{2dim(X)-n}$. As $Frob_W$ is injective, we have that $F^*$ is injective. Thus, $(H^n, F^*)$ is an F-crystal. We will denote the F-isocrystal $H^n_{crys}(X/W) \otimes K$ by $H^n_{crys}(X/K)$.
\item The F-isocrystal $K(1) : = (K, Frob_K/p)$. Similarly, one has the F-isocrystal $K(n) := (K, Frob_K/p^n)$ for all $n \in \Z$. Moreover, for any F-isocrystal $V$ and $n \in \Z$, we denote by $V(n)$ the F-isocrystal $V \otimes K(n)$. 
\end{enumerate}

 Recall that the category of F-crystals over $k$ up to isogeny is semi-simple and the simple objects are the F-crystals: 
\begin{equation*}
M_{\alpha} = ((\Z_p[T])/(T^s-p^r)) \otimes_{Z_p} W(k), (\text{mult. by} \ T)\otimes Frob_W),
\end{equation*} 
for $\alpha = r/s \in \Q_{\geq 0}$ and $r$, $s$ non-negative coprime integers. 
This is a theorem of Dieudonn\'e-Manin. Note that the rank of the F-crystal $M_{\alpha}$ is $s$. We call $\alpha$ the \textbf{slope} of the F-crystal $M_{\alpha}$. 

\begin{definition} Let $(M, \phi)$ be an F-crystal over $k$ and let 
$$
(M, \phi) \sim^{isogeny} \oplus_{\alpha \in \Q_{\geq 0}} M_{\alpha}^{n_{\alpha}}
$$
be its decomposition up to isogeny. Then the elements of the set 
$$
\{\alpha \in \Q_{\geq 0}| n_{\alpha} \neq 0 \}
$$
are called the \textbf{slopes} of $(M, \phi)$. For every slope $\alpha$ of $(M, \phi)$, the integer $\lambda_{\alpha} := n_{\alpha} \cdot rank_W M_{\alpha}$ is called the \textbf{multiplicity} of the slope $\alpha$.
\end{definition}

\begin{remark}
In case $(M, \phi)$ is an F-crystal over a perfect field $k$ (rather than being algebraically closed as assumed above), we define its slope and multiplicities to be that of the F-crystal $(M, \phi) \otimes _{W(k)}W(\bar{k})$, where $\bar{k}$ is an algebraic closure of $k$. 
\end{remark}

We still keep our assumption of $k$ being an algebraically closed field of positive characteristic. 

The above classification result of Dieudonn\'e-Manin is more general. Any F-isocrystal $V$ with bijective $\phi_V$ is isomorphic to a direct sum of F-isocrystals
$$
(V_{\alpha} := K[T]/(T^s-p^r), (\text{mult. by $T$}) \otimes Frob_K), 
$$   
for $\alpha = r/s \in \Q$. The dimension of $V_{\alpha}$ is $s$ and we call $\alpha$ the \textbf{slope} of $V_{\alpha}$.

\begin{definition}[Height]
The \textbf{height} of a K3 surface $X$ over $k$ is the sum of multiplicities of slope strictly less than 1 part of the F-crystal $H^2_{crys}(X/W)$. In other words, the dimension of the subspace of slope strictly less than one of the F-isocrystal $H^2_{crys}(X/K)$, which is $\text{dim} (H^2_{crys}(X/K)_{[0,1)} := \oplus_{\alpha_i < 1}V_{\alpha_i}^{n_{\alpha_i}})$.
\end{definition}

If for a K3 surface $X$ the $\text{dim} (H^2_{crys}(X/K)_{[0,1)}) =0$, then we say that the height of $X$ is infinite.

Supersingular K3 surfaces (i.e., K3 surfaces with infinite height) also have an equivalent description that their Picard rank is 22 (see \cite{Liedtke} Theorem 4.8).  We will be discussing more about F-crystals later in Appendix \ref{F-crystalchap}. 

Lastly, we state the theorem by Deligne about lifting K3 surfaces which will be used a lot in the theorems that follow.

Let $X_0$ be a K3 surface over a field $k$ of characteristic $p > 0$. 
\begin{definition}[Lift of a K3 surface]
A \textbf{lift} of a K3 surface $X_0$ to characteristic $0$ is a smooth projective scheme $X$ over $R$, where $R$ is a discrete valuation ring such that $R/\mathfrak{m} = k$, $K := \text{Frac} (R) $ is a field of characteristic zero, the generic fiber of $X$, denoted $X_K$, is a K3 surface and the special fiber is $X_0$. 
\end{definition}

\begin{theorem}[Deligne \cite{Deligne} Theorem 1.6, corollary 1.7, 1.8] \label{Delignelifting}
Let $X_0$ be a K3 surface over a field $k$ algebraically closed of characteristic $p > 0$.  Let $L_0$ be an ample line bundle on $X_0$. Then there exists a finite extension $T$ of $W(k)$, the Witt ring of $k$, such that there exists a deformation of $X_0$ to a smooth proper scheme $X$ over $T$ and an extension of $L_0$ to an ample line bundle $L$ on $X$.  
\end{theorem}

Consider the situation where we have a lift of a K3 surface, i.e., let $X_0$ be a K3 surface over a field of characteristic $p >0$ and $X$ a lift over $S = \Spec(R)$ as defined above. The de Rham cohomology of $X/S$, $H^*_{DR}(X/S)$ is equipped with a filtration induced from the Hodge to de Rham spectral sequence:
$$
E_1^{i,j} = H^j(X, \Omega^i_{X/S}) \Rightarrow H^*_{DR}(X/S)
$$
For a construction of this spectral sequence, see \cite{Grothendieck} III-0 11.2. We call this filtration on $H^2_{DR}(X/S)$ the Hodge filtration. Using the comparison isomorphism between the crystalline cohomology of the special fiber and the de Rham cohomology of $X$ \cite{BOcrys} 7.26.3, 
$$
H^i_{crys}(X_0/W(k)) \otimes R \cong H^i_{DR}(X/S),
$$
we get a filtration on the crystalline cohomology, also called \textbf{the Hodge filtration}. This Hodge filtration on the crystalline cohomology depends on the choice of a lift of $X_0$.

Next we discuss about the Moduli space of sheaves on a K3 surface as these spaces turn out to play a very important role in the theory of derived equivalences of K3 surfaces. We introduce the moduli stack of sheaves on a K3 surface and show that it's a $\mu_r-$Gerbe under some numerical conditions. We will try to keep the exposition here characteristic independent and in case of characteristic restrictions we will mention them as necessary. Moreover, in the case of a K3 surface defined over a field we will not assume the field to be algebraically closed and in general, for a relative K3 surface, we will work with a spectrum of a  mixed characteristic discrete valuation ring as the base scheme. The main references for this section are  \cite{L2} Section 2.3.3 and \cite{LO} Section 3.15. We refer the reader to \cite{gomez}, for a comparison between the moduli stack point of view and that of more classical moduli functors.  For an introduction to theory of gerbes we refer the reader to  \cite{OlssonAS}.

\begin{remark}
The point of view of moduli stacks offers us the benefit that in this way it becomes more natural to generalize this theory to derived schemes and derived stacks, where the role of Artin representability theorem will be taken up by Lurie's Artin Representability theorem \cite{LurieFM} and those of the Hilbert and Quot schemes by their derived versions. 
\end{remark} 

Before proceeding to the definition of moduli stacks of sheaves that we will be working with, let us also recall the notion of (Gieseker) semistability for coherent sheaves (for details see \cite{HL}, Section 1.2): Let $X$ be a projective scheme over a field $k$. The Euler characteristic of a coherent sheaf $\mathcal{F}$ is $\chi(\mathcal{F})= \sum (-1)^i h^i(X, \mathcal{F})$. If we fix an ample line bundle $\O(1)$ on $X$, then the Hilbert polynomial $P(\mathcal{F})$ given by $n \mapsto \chi(\mathcal{F} \otimes \O(n))$ can be uniquely written in the form 
$$
P(\mathcal{F}, n) = \sum_{i =0}^{dim(\mathcal{F})} \alpha_i(\mathcal{F}) m^i/i! ,
$$ 
with integral coefficients $\alpha_i(\mathcal{F})$. We denote by $p(\mathcal{F},n) := P(\mathcal{F}, n)/ \alpha_{dim(\mathcal{F})}(\F)$, the \textbf{reduced Hilbert polynomial of $\F$}.

\begin{definition}[Semistability]
A coherent sheaf $\F$ of dimension $d$ is \textbf{semistable} if $\F$ has no nontrivial proper subsheaves of strictly smaller dimension  and for any subsheaf $\E \subset \F$, one has $p(\E) \leq p(\F)$. It is called \textbf{stable} if  for any proper subsheaf the inequality is strict. 
\end{definition}

\begin{remark}
The ordering on polynomials is the lexicographic ordering of the coefficients. 
\end{remark}

\begin{definition}[Mukai vector]  \label{defMukaivector}
For a smooth projective $X$ over $k$, given a perfect complex $E \in D(X)$, where $D(X)$ is the derived category of coherent sheaves on $X$, we define the \textbf{Mukai vector} of $E$ to be 
\begin{equation*}
v(E) := ch(E) \sqrt{td_X} \in A^*(X)_{num, \Q}.
\end{equation*}
Here, $ch(-)$ denotes the Chern class map, $td_X$ is the Todd genus and $A^*(X)_{num, \Q}$ is the numerical Chow group of $X$ with rational coefficients. 
\end{definition}

For $X$ a K3 surface over $k$, the Mukai vector of  a complex is given by (see \cite{HuyFM} Chapter 10):
\begin{equation*}
v(E) = (\text{rank} (E), c_1(E), \text{rank} (E) + c_1(E)^2/2 - c_2(E)).
\end{equation*}

Let $X$ be a projective scheme over $k$ and $h$ an ample line bundle.
\begin{definition}[Moduli Stack]
The \textbf{moduli stack of semistable sheaves}, denoted $\mathfrak{M}_h^{ss}$, is defined as follows:
\begin{equation*}
\begin{split}
\mathfrak{M}_h^{ss}: (Sch/k) &\ra (\text{groupoids}) \\
 S &\mapsto \{ \mathcal{F} | \mathcal{F} \ \text{an $S$-flat coherent sheaf on $X \times S$ with semistable fibers} \}.
\end{split}
\end{equation*}
Similarly, the \textbf{moduli stack of stable sheaves} can be defined by replacing semistable above with stable and we denote it by $\mathfrak{M}^s_h$. 
\end{definition}

If we fix a vector $v \in A^*(X)_{num, \Q}$, we get an open and closed substack $\mathfrak{M}^{ss}_h(v)$ classifying semistable sheaves on $X$ with Mukai vector $v$.

The following result has been proved by Lieblich \cite{L2}, for the more general case of moduli of twisted sheaves. Restricting to the case of semistable sheaves without any twisting a simpler argument is given in \cite{SrivasPhD} Theorem 2.30. 

\begin{theorem} \label{agstack}
The stack $\mathfrak{M}^{ss}_h$ is an algebraic stack and the stack $\mathfrak{M}^{ss}(v)$ is an algebraic substack of finite type over $k$. 
\end{theorem}

\begin{remark} 
Recall that the Mukai vector $v$ for a sheaf on a K3 surface determines its Hilbert polynomial and its rank as well.
\end{remark}

Moreover, the stack $\mathfrak{M}_h^{ss}(v)$ contains an open substack of geometrically stable points (see Footnote \ref{Footnote5}) denoted $\mathfrak{M}_h^s(v)$. 

\begin{theorem}[\cite{L2}, Lemma 2.3.3.3 and Proposition 2.3.3.4 or \cite{SrivasPhD} Theorem 2.34]  \label{coarseMS2}
The algebraic stack $\mathfrak{M}_h^{s}(v)$ admits a coarse moduli space.
\end{theorem}

The above theorem also implies that $\mathfrak{M}_h^{s}(v)$ is a $\mu_r$-gerbe. Indeed, recall that for an algebraic stack $\mathcal{X}$, the morphism $\mathcal{X} \ra Sh(\mathcal{X})$ is a $\mu_r$-gerbe if and only if $Sh(\mathcal{X})$ is isomorphic to the final object in the topos $Sh(\mathcal{X})$ and the automorphism sheaf is isomorphic to $\mu_r$ (see, for example, \cite{L2} 2.1.1.12). The first condition is obvious and for the second we just need to compute the automorphism sheaf of any point in $\mathfrak{M}_h^{s}(v)$, but this just corresponds to finding out the automorphisms of a semistable sheaf with fixed determinant line bundle (note that the Mukai vector determines the determinant line bundle) and fixed rank $r$, which turns out to be $\mu_r$. This gerbe corresponds via \cite{OlssonAS}, Theorem 12.2.8 to a class $\alpha_r$ in  $H^2(X, \mu_r)$. 

The Kummer exact sequence 
$$
0 \ra \mu_r \ra \mathbb{G}_m \ra \mathbb{G}_m \ra 0 
$$ 
induces a long exact sequence of group cohomology, giving us a map $H^2(X, \mu_r) \ra H^2(X, \mathbb{G}_m)$. The image of the class $\alpha_r$ in $H^2(X, \mathbb{G}_m)$ gives us a corresponding $\mathbb{G}_m$-gerbe, again using \cite{OlssonAS}, Theorem 12.2.8. This class is the obstruction to the existence of the universal bundle on $Sh(\mathfrak{M}_h^s(v)) \times X$.

\begin{theorem}[Mukai-Orlov]
Let $X$ be a K3 surface over a field $k$. 
\begin{enumerate}
\item Let $v \in A^*(X)_{num, \Q}$ be a primitive element with $v^2 = 0$ (with respect to the Mukai pairing\footnote{The Mukai pairing is just an extension of the intersection pairing, defined as follows: let $(a_1, b_1, c_1) \in A^*(X)_{num, \Q}$ and $(a_2, b_2, c_2) \in A^*(X)_{num, \Q}$, then the Mukai pairing is $<(a_1, b_1, c_1), (a_2, b_2, c_2)> = b_2\cdot b_1 - a_1\cdot c_2 - a_2 \cdot c_1 \in A^2(X)_{num, \Q}$.}) and positive degree $0$ part\footnote{The degree zero part just means the $A^0(X)_{num, \Q}$ term in the representation of the Mukai vector in $A^*(X)_{num, \Q}$.}. Then $\mathfrak{M}^{ss}_h(v)$is non-empty.
\item If, in addition, there is a complex $P \in D(X)$ with Mukai vector $v'$ such that $<v, v'> = 1$, then every semistable sheaf with Mukai vector $v$ is locally free and geometrically stable\footnote{\label{Footnote5}A coherent sheaf $\F$ is \textbf{geometrically stable} if for any base field extension $l / k$,  the pullback $ \F \otimes_k l$ along $X_l = X \times_k \Spec(l) \ra X$ is stable.}, in which case $\mathfrak{M}^{ss}_h(v)$ is a $\mu_r$-gerbe for some $r$, over a smooth projective surface $M_h(v)$\footnote{We will denote this moduli space later as $M_X(v)$ to lay emphasis that it is the moduli space of stable sheaves  over $X$.} such that the associated $\mathbb{G}_m$-gerbe is trivial.   
\end{enumerate}
\end{theorem}

\begin{remark}
\begin{enumerate}
\item Note that the triviality of the $\mathbb{G}_m$-gerbe is equivalent to the existence of a universal bundle over $X \times M_h(v)$, also see \cite{LO} Remark 3.19.
\item See Remark 6.1.9 \cite{HL} for a proof that under the assumption of the above Theorem part (2), any semistable sheaf is locally free and geometrically stable. 
\end{enumerate}
\end{remark}

\begin{proof}
The non-emptiness follows from \cite{HuyLect} Chapter 10 Theorem 2.7 and \cite{LO} Remark 3.17. For the construction of the universal bundle, one has to ,in the end, actually use GIT again. For a proof see \cite{HuyLect} Chapter 10 Proposition 3.4 and \cite{HL} Theorem 4.6.5 (this is from where we have the numerical criteria, in particular, also see \cite{HL} Corollary 4.6.7.). 
\end{proof}

We generalize our moduli stack to the relative setting. Let $X_S$ be a flat projective scheme over $S$ with an ample line bundle $h$. (The case of $S = \Spec(R)$ for $R$ a discrete valuation ring of mixed characteristic, will be of most interest to us.) 

\begin{definition}[Relative Moduli Stack] The \textbf{relative moduli stack of semi-stable sheaves}, denoted $\mathfrak{M}_h^{ss}$, is defined as follows:
\begin{equation*}
\begin{split}
\mathfrak{M}_h^{ss}: (Sch/S) &\ra (\text{groupoids}) \\
 T &\mapsto \{ \mathcal{F} | \mathcal{F} \ \text{$T$-flat coherent sheaf on $X \times_S T$ with semistable fibers} \}. 
\end{split}
\end{equation*}
The \textbf{relative moduli stack of stable sheaves} can be defined similarly and we denoted it by $\mathfrak{M}^s_h$. 
\end{definition} 
 
The following theorem shows the existence of the fine moduli space for the relative moduli stack, when $X_R$ is a relative K3 surface over a mixed characteristic discrete valuation ring, under some numerical conditions. Recall that the condition of flatness is going to be always satisfied in our relative K3's case by definition as they are smooth. The relative stack can be proived to be an algebraic stack using arguments similar to the ones used for proving Theorem \ref{agstack}. Moreover, all the results above about the moduli stack hold also for the relative stack. So, there exists a coarse moduli space (Compare from footnote 1 in \cite{HuyLect} Chapter 10 or \cite{HL} Thm 4.3.7, the statement there is actually weaker as we do not ask for morphism of $k$-schemes, which is not going to be possible for mixed characteristic case. So, for the mixed characteristic case one replaces, in the GIT part of the proof, the quot functor by its relative functor, which is representable in this case as well \cite{NNQuot} Theorem 5.1). Moreover, the non-emptiness results also remain valid in mixed characteristic setting and we have:

\begin{theorem}[Fine relative Moduli Space] \label{finerelativemoduli} 
Let $X_V$ be a relative K3 surface over a mixed characteristic discrete valuation ring $V$ with $X$ as a special fiber over $\Spec(k)$
\begin{enumerate}
\item Let $v \in A^*(X)_{num, \Q}$\footnote{Note that in the mixed characteristic setting, for any complex $E_V \in D^b(X_V)$ we define its Mukai vector to be just the Mukai vector of $E := E_V \otimes_V k$ in $A^*(X)_{num, \Q}$. This definition makes sense as $X_V \ra V$ is flat.} be a primitive element with $v^2 = 0$ (with respect to the Mukai pairing) and positive degree $0$ part\footnote{The degree zero part just means the $A^0(X)_{num, \Q}$ term in the representation of the Mukai vector in $A^*(X)_{num, \Q}$.}. Then, $\mathfrak{M}^{ss}_h(v)$, the sub-moduli stack of $\mathfrak{M}^{ss}_h$ with fixed Mukai vector $v$, is non-empty.
\item If, in addition, there is a complex $P \in D(X_V)$ with Mukai vector $v'$ such that $<v, v'> = 1$, then every semistable sheaf with Mukai vector $v$ is locally free and stable, in which case $\mathfrak{M}^{ss}_h(v)$ is a $\mu_r$-gerbe for some $r$, over a smooth projective surface $M_h(v)$ such that the associated $\mathbb{G}_m$-gerbe is trivial.   
\end{enumerate}
\end{theorem}

With this we conclude our exposition on moduli stacks and spaces of sheaves. We now give a summary of selected results on derived equivalences of a K3 surfaces for both positive characteristic and characteristic zero. We begin by a general discussion on derived equivalences and then specialize to different characteristics.   

Let $X$ be a K3 surface over a field $k$ and let $D^b(X)$ be the bounded derived category of coherent sheaves of $X$. We refer the reader to \cite{HuyFM} for a quick introduction to derived categories and the textbooks \cite{GM}, \cite{KS} for details. 

\begin{definition} 
Two K3 surfaces $X$ and $Y$  over $k$ are said to be \textbf{derived equivalent} if there exists an exact equivalence $D^b(X) \simeq D^b(Y)$ of the derived categories as triangulated categories\footnote{We don't need to start with $Y$ being a K3 surface, this can be deduced as a consequence by the existence of an equivalence on the level of derived categories of varieties, see \cite{HuyFM} Chapter 4 and Chapter 6 and Chapter 10 and \cite{BBR} Chapter 2 for the properties preserved by derived equivalences.  However, note that Orlov's Representability Theorem \ref{ORT} is used in some proofs.}. 
\end{definition}

\begin{definition}[Fourier-Mukai Transform]
For a perfect complex  $\mathcal{P} \in D^b(X \times Y)$, the \textbf{Fourier-Mukai transform} is a functor of the derived categories which is defined as follows:
\begin{equation*}
\begin{split}
\Phi_P: D^b(X) &\ra D^b(Y)  \\
\mathcal{E} &\mapsto \mathbb{R}p_{Y*}( (p_X^* \mathcal{E}) \otimes^{\mathbb{L}} \mathcal{P}),
\end{split}
\end{equation*} 
where $p_X, p_Y$ are the projections from $X \times Y$ to the respective  $X$ and $Y$.
\end{definition}

\begin{remark}
The boundedness of the derived categories: We restrict to the bounded derived categories as it allows us to employ cohomological methods to study derived equivalences, as explained below. 
\end{remark}

For details on the properties of Fourier-Mukai transform see \cite{HuyFM} Chapter 5. Note that not every Fourier-Mukai transform induces an equivalence. The only general enough criteria available to check whether the Fourier Mukai transform induces a derived equivalence is by Bondol-Orlov, see for example, \cite{HuyLect} Chapter 16 Lemma 1.4, Proposition 1.6 and Lemma 1.7. In case the Fourier-Mukai transform is an equivalence, we have the following definition:
\begin{definition}
A K3 surface $Y$ is said to be a \textbf{Fourier Mukai partner} of $X$ if there exists a Fourier-Mukai transform between $D^b(X)$ and $D^b(Y)$ which is an equivalence. We denote by  $FM(X)$ the set of isomorphism classes of Fourier Mukai Partners of $X$ and by $|FM(X)|$ the cardinality of the set, which is called the \textbf{Fourier Mukai number} of $X$.
\end{definition}

We state here the most important result in the theory of Fourier-Mukai transforms and derived equivalences.

\begin{theorem}[Orlov, \cite{HuyFM} Theorem 5.14] \label{ORT}
Every equivalence of derived categories for smooth projective varieties is given by a Fourier Mukai transform. More precisely, let $X$ and $Y$ be two smooth projective varieties and let 
\begin{equation*}
F: D^b(X) \ra D^b(Y)
\end{equation*} 
be a fully faithful exact functor. If $F$ admits right and left adjoint functors, then there exists an object $P \in D^b(X \times Y)$ unique up to isomorphism such that $F$ is isomorphic to $\Phi_P$.
\end{theorem}

\begin{remark}
This theorem allows us to restrict the collection of derived equivalences to a smaller and more manageable collection of Fourier-Mukai transforms, which will be studied via cohomological descent.  
\end{remark}

Any Fourier Mukai transform, $\Phi_P$, descends from the level of the derived categories to various cohomological theories $(H^*_{\ldots}( \ ))$, as
\begin{center}
$\begin{CD}
 D^b(X)@>\text{$\mathcal{E}  \mapsto \mathbb{R} p_{Y*}((\mathbb{L}p_X^* \mathcal{E}) \otimes^{\mathbb{L}} P) $}>>D^b(Y)\\
 @VVch( \ )\sqrt{td_{X}} V @VVch( \ )\sqrt{td_{Y}} V\\
 H^*(X)@>\text{$\alpha \mapsto p_{Y*}\big(( p_X^*\alpha) \cdot ch(P) \sqrt{td_{X\times Y}}\ \big) $}>> H^*(Y),
\end{CD}$\\
\end{center}
where $ch( \ )$ is the total Chern character and $td_{X}$ is the Todd genus of $X$. This descent provides a way to study the Fourier Mukai partners of $X$ using cohomological methods. For details see \cite{HuyFM} Section 5.2 and \cite{LO} Section 2. \\

In characteristic $0$ (mostly over $\C$, see remark \ref{char0} below), we will use the singular cohomology along with $p/l$-adic/\'etale cohomology  and in characteristic $p >0$, we will use crystalline cohomology or $l$-adic etale cohomology.  In the mixed characteristic setting, we will be frequently using a different combination of cohomologies along with their comparison theorems from $p$-adic Hodge theory. 

\begin{remark}
The Orlov Representability Theorem \ref{ORT} works only for smooth projective varieties, so when we work with relative schemes we will restrict from the collection of derived equivalences and work only with the subcollection of Fourier-Mukai transforms. 
\end{remark}

Over the field of complex numbers, Mukai and Orlov provide the full description of the set $FM(X)$ as:
\begin{theorem} [Mukai \cite{Mu}, Theorem 1.4 and Theorem 1.5, \cite{Orlov1}]  \label{DerivedHodge} Let $X$ be a K3 surface over $\C$. Then the following are equivalent:
\begin{enumerate}
\item There exists a Fourier-Mukai transform  $\Phi: D^b(X) \cong D^b(Y)$ with kernel $\mathcal{P}$.
\item There exists a Hodge isometry $f: \tilde{H}^*(X , \mathbb{Z}) \rightarrow \tilde{H}^*(Y , \mathbb{Z}) $, where $\tilde{H}^*( \ , \mathbb{Z})$ is the singular cohomology of the corresponding analytic space and is compared with the de Rham cohomology of the algebraic variety $X$ which comes with a Hodge filtrations and Mukai pairing \footnote{The Mukai pairing is just an extension of the intersection pairing, defined as follows: let $(a_1, b_1, c_1) \in \tilde{H}^*(X , \mathbb{Z})$ and $(a_2, b_2, c_2) \in \tilde{H}^*(X , \mathbb{Z})$, then the Mukai pairing is $<(a_1, b_1, c_1), (a_2, b_2, c_2)> = b_2 \cdot b_1 - a_1 \cdot c_2 - a_2 \cdot c_1 \in H^4(X, \Z)$.}.  
\item There exists a Hodge isometry $f: T(X) \simeq T(Y)$ between their transcendental lattices.
\item  $Y$ is a two dimensional fine compact moduli space of stable sheaves on $X$ with respect to some polarization on $X$, i.e., $Y \cong M_X(v)$ for some Mukai vector $v \in A^*(X)_{num, \Q}$ \footnote{Compare from Definition \ref{defMukaivector}.}.
\item  There is an isomorphism of Hodge structures between $H^2(M_X(v), \Z)$ and $v^{\perp}/\Z v$ which is compatible with the cup product pairing on $H^2(M_X(v), \Z)$ and the bilinear form on $v^{\perp}/\Z v$ induced by that on the Mukai lattice $\tilde{H}^*(X, \Z)$.  
\end{enumerate}
\end{theorem}

The following result is the \'etale version of the Mukai-Orlov cohomological version of decription of derived equivalences of K3 surfaces over $\C$. 

\begin{proposition}[p-adic \'etale cohomology version] \label{etaleversion}
If $X$ and $Y$ are derived equivalent K3 surfaces, then there is an isomorphism between $H^2_{\acute{e}t} (M_X(v), \Z_p)$ and $v^{\perp}/ \Z_p v$, (see footnote \footnote{We are abusing the notation here: The Mukai vector is now considered as an element of $H^*_{\acute{e}t}(X, \Z_p)$ and $v^{\perp}$ is the orthogonal complement of $v$ in $H^*_{\acute{e}t}(X, \Z_p)$ with respect to Mukai pairing. Thus, $v^{\perp}$ is a $\Z_p$ lattice. Then we mod out this lattice by the $\Z_p$ module generated by $v$.}), which is compatible with the cup product pairing on $H^2_{\acute{e}t} (M_X(v), \Z_p)$ and the bilinear form on $v^{\perp}/\Z_p v$ induced by that on the Mukai lattice $\tilde{H}^*(X, \Z_p)$, where $p$ is a prime number and $\Z_p$ is the ring of $p$-adic integers.  
\end{proposition}

\begin{proof}
This follows from Artin's Comparison Theorem \cite{SGA4} Tome III, Expos\'e 11, Th\'eor\`eme 4.4 between \'etale and singular cohomology and the theorem above.
\end{proof}

\begin{proposition}[\cite{HuyLect} Proposition 3.10] \label{finitechar0}
Let $X$ be a complex projective K3 surface, then $X$ has only finitely many Fourier-Mukai partners, i.e.,
$$
|FM(X)| < \infty.
$$ 
\end{proposition}

\begin{remark} \label{finiteChar0} 
The above result is also true for any algebraically closed field of characteristic $0$. Indeed, if $X$ and $Y$ are two K3 surfaces over a field $K$ algebraically closed and characteristic $0$, we have $X \cong Y \Leftrightarrow X_{\C} \cong Y_{\C}$. One way is obvious via base change and for the other direction we just need to show that every isomorphism $X_{\C} \cong Y_{\C}$ comes from an isomorphism $X \cong Y$. To define an isomorphism only finitely many equations are needed, so we can assume that the isomorphism is defined over $A$, a finitely generated $K$-algebra (take $A$ to be the ring $K[a_1, \ldots , a_n]$, where $a_i$ are the finitely many coefficients of the finitely many equations defining our isomorphism). Thus, we have have our isomorphism defined over an affine scheme, $X_A \cong Y_A$, where $X_A := X \times_K \Spec(A)$ (resp. $Y_A := Y \times_K \Spec(A)$). As $K$ is algebraically closed, any closed point $t \in \Spec(A)$ has residue field $K$. Now taking a $K$-rational point will give us our required isomorphism. 

This gives us a natural injection: 
\begin{equation*}
\begin{split}
FM(X) &\hookrightarrow FM(X_{\C})\\
Y &\mapsto Y_{\C}. 
\end{split}
\end{equation*}
Hence, we have $|FM(X)| \leq |FM(X_{\C})| < \infty$. 
\end{remark}

Let $S=NS(X)$ be the N\'eron-Severi lattice of $X$. The following theorem gives us the complete counting formula for Fourier-Mukai partners of a K3 surface.

\begin{theorem}[Counting formula \cite{HLOY}] \label{CountingformulaChar0}
Let $\mathcal{G}(S) = \{S_1=S,S_2, \ldots S_m \}$ be the set of isomorphism classes of lattices with same signature and discriminant as $S$. Then
\begin{equation*}
|FM(X)| = \sum_{j=1}^{m}|Aut(S_j) \backslash Aut(S_j^*/S_j ) / O_{Hdg}(T(X))| < \infty.
\end{equation*}
\end{theorem}

The relation with the class number $h(p)$ of $\mathbb{Q}(\sqrt{-p})$, for a prime p, is:

\begin{theorem} [\cite{HLOY} Theorem 3.3] \label{classnocount}
Let the $rank \ NS(X) = 2$  for X, a K3 surface, then $\det NS(X) = -p$ for some prime $p$, and $|FM(X)|=(h(p) +1)/2$.
\end{theorem}

\begin{remark}
The surjectivity of period map (\cite{HuyLect} Theorem 6.3.1)  along with \cite{HuyLect} Corollary 14.3.1 implies that there exists a K3 with Picard rank $2$ and discriminant $-p$, for each prime $p$ (see \cite{HLOY} Remark after Theorem 3.3).\\
\end{remark}

We now describe the known results about the derived autoequivalence group $Aut(D^b(X))$ for a K3 surface over $\C$. Observe that Theorem \ref{DerivedHodge} implies that we have the following natural map of groups:
$$
Aut(X) \hookrightarrow Aut(D^b(X)) \ra O_{Hdg}(\tilde{H}^*(X, \Z)).
$$ 
The following theorem gives a description of the second map:
\begin{theorem}[\cite{HLOY}, \cite{Ploog}]
Let $\vp$ be a Hodge isometry of the Mukai lattice $\tilde{H}^*(X, \Z)$ of a K3 surface $X$, i.e. $\vp \in O_{Hdg}(\tilde{H}^*(X, \Z))$. Then there exists an autoequivalence 
\begin{equation}
\Phi_{E} : D^b(X) \ra D^b(X)
\end{equation}
with $\Phi^H_E= \vp \circ ( \pm id_{H^2}): \tilde{H}^*(X, \Z) \ra \tilde{H}^*(X, \Z)$. In particular, the index of image
\begin{equation}
Aut(D^b(X)) \ra O_{Hdg}(\tilde{H}^*(X, \Z))
\end{equation}
is at most 2. 
\end{theorem}

On the other hand, it has been shown that
\begin{theorem}[\cite{Huy}]
The cone-inversion Hodge isometry $id_{H^0 \oplus H^4}\oplus -id_{H^2}$ on $\tilde{H}^*(X, \Z)$ is not induced by any derived auto-equivalence. 
In particular, the index of image 
\begin{equation}
Aut(D^b(X)) \ra O_{Hdg}(\tilde{H}^*(X, \Z))
\end{equation}
is exactly 2. 
\end{theorem}

\begin{remark} \label{char0} [\cite{HuyLect} 16.4.2]
The above results have been shown for K3 surfaces over $\C$ only but the results are valid for K3 surfaces over any algebraically closed field of characteristic $0$, in the sense made precise below. The argument goes as follows: We reduce the case of $char(k)= 0$ to the case of $\C$. 
We begin by making the observation that every K3 surface $X$ over a field $k$ is defined over a finitely generated subfield $k_0$, i.e., there exists a K3 surface $X_0$ over $k_0$ such that $X := X_0 \times_{k_0} k$. Similarly, if $\Phi_P : D^b(X) \ra D^b(Y)$ is a Fourier Mukai equivalence, then there exists a finitely generated field $k_0$ such that $X, Y $ and $P$ are defined over $k_0$. Moreover, the $k_0$- linear Fourier-Mukai transform induced by $P_0$, $\Phi_{P_0}: D^b(X_0) \ra D^b(Y_0)$ will again be a derived equivalence (use, for example, the criteria \cite{HuyFM} Proposition 7.1 to check this.). 

Now assume that $k_0$ is algebraically closed. Note that any Fourier-Mukai kernel which induces an equivalence $\Phi_{P_0}: D^b(X_0) \xrightarrow{\sim} D^b(X_0)$ is rigid, i.e. $\text{Ext}^1(P_0,P_0) = 0$ (see \cite{HuyLect} Proposition 16.2.1),  thus any Fourier-Mukai equivalence 
$$
\Phi_{P}: D^b(X_0 \times_{k_0} k) \xrightarrow{\sim} D^b(X_0 \times_{k_0} k)
$$
descends to $k_0$  (see for example \cite{HuyLect} Lemma 17.2.2 for the case of line bundles, the general case follows  similarly\footnote{In the general case we sketch the proof: Use the moduli stack of simple universally gluable perfect complexes over $X_0 \times X_0 / k_0$, denoted $s\mathcal{D}_{X_0 \times X_0/k_0}$, as defined in Definition \ref{Def4.9}. From the arguments following the definition, it is an algebraic stack which admits a coarse moduli algebraic space $sD_{X_0 \times X_0/k_0}$. Note that for any $k_0$ point $P_0$ which induces an equivalence, the local dimension of the coarse moduli space is zero as the tangent space is a subspace of $\text{Ext}^1(P_0,P_0) = 0$ (see, for example, \cite{L} 3.1.1 or proof of \cite{LO} Lemma 5.2) and the coarse moduli space is also smooth. The smoothness follows from the fact that the deformation of the complex is unobstructed (see, for example, \cite{SP} Tag 03ZB and Tag 02HX) in equi-characteristic case as one always has a trivial deformation. Indeed, let $A$ be any Artinian local $k$-algebra, then pullback  along the structure morphism $\Spec(A) \ra \Spec(k)$ gives a trivial deformation of $X \times X$ and also a trivial deformation of any complex on $X \times X$.   Thus, we can repeat the argument as in \cite{HuyLect} Lemma 17.2.2 as now the image of the classifying map $f: \Spec(A) \ra sD_{X_0 \times X_0/k_0}$ is constant (In the notation of \cite{HuyLect} Lemma 17.2.2).}). Hence, for a K3 surface $X_0$ over the algebraic closure $k_0$ of a finitely generated field extension of $\Q$ and for any choice of an embedding $k_0 \hookrightarrow \C$, which always exists, one has 
\begin{equation*}
Aut(D^b(X_0 \times_{k_0} k))\cong  Aut(D^b(X_0)) \cong Aut(D^b(X_0 \times_{k_0} \C)).
\end{equation*}
In this sense, for K3 surfaces over algebraically closed fields $k$ with $char(k) = 0$, the situation is identical to the case of complex K3 surfaces.
\end{remark}

We can now write down the following exact sequence: For $X$ a projective complex $K3$ surface one has
\begin{equation}
0 \ra \text{Ker} \ra \text{Aut}(D^b(X)) \ra \text{O}_{\text{Hdg}}(\tilde{H}^*(X, \Z))/\{\mathfrak{i}\} \ra 0,
\end{equation}
where $\tilde{H}^*(X, \Z)$ is the cohomology lattice with Mukai pairing and extended Hodge structure, and $\text{O}_{\text{Hdg}}(-)$ is the group of Hodge isometries, $\mathfrak{i}$ is the cone inversion isometry ${\text Id}_{H^0 \oplus H^4} \oplus -{\text Id}_{H^2}$.

\begin{remark}
\begin{enumerate}
\item The structure of the kernel of this map has been described only in the special case of a projective complex $K3$ surface with $\text{Pic}(X) = 1$ in \cite{BB}. (For a discussion about the results in non-projective case see \cite{Huy2}.) However, Bridgeland  in \cite{Bri2} (Conjecture 1.2) has conjectured that this kernel can be described as the fundamental group of an open subset of $H^{1,1} \otimes \C$. Equivalently, the conjecture says that the connected component of the stability manifold (see \cite{Bri1}, \cite{Bri2} for the definitions) associated to the collection of the stability conditions on $D^b(X)$ covering an open subset of $H^{1,1} \otimes \C$ is simply connected. The equivalence of the two formulations follows from a result of Bridgeland (\cite{Bri2} Theorem 1.1), which states that the kernel acts as the group of deck transformations of the covering of an open subset of $H^{1,1} \otimes \C$ by a connected component of the stability manifold. Bayer and Bridgeland \cite{BB} have verified the conjecture in the special cases of $\text{Pic}(X) =1$, (see \cite{Huy2} for the non-projective case).

\item Note that Bridgeland defines the stability conditions for any small triangulated category, so even in the case of derived category of a $K3$ over a field of positive characteristic we can associate the stability manifold which will still be a complex manifold (\cite{Bri2}, Remark 3.2).
\end{enumerate}
\end{remark}

Lastly,  we state the main results on derived equivalences of K3 surfaces over an algebraically closed field of positive characteristic known so far. For generalizations of some results to non-algebraically closed fields of positive characteristic see \cite{MatthewWard}.   

In case, $char (k) = p >2$, Lieblich-Olsson \cite{LO}, proved the following:
\begin{theorem}[\cite{LO}, Theorem 1.1] \label{LOmainthm}
Let $X$ be a K3 surface over an algebraically closed field $k$ of positive characteristic $\neq 2$.
\begin{enumerate}
\item If $Y$ is a smooth projective k-scheme with $D^b(X) \cong D^b(Y)$, then Y is a K3 surface isomorphic to a fine moduli space of stable sheaves.   
\item There exists only finitely many smooth projective k-schemes $Y$ with $D^b(X) \cong D^b(Y)$. If $X$ has $ rank \ NS(X) \geq 12$, then $D^b(X) \cong D^b(Y)$ implies that $X \cong Y$.  In particular, any supersingular K3 surface is determined up to isomorphism by its derived category.  
\end{enumerate}
\end{theorem}

\begin{remark}
One of the open questions is to have a cohomological criteria for derived equivalent K3 surfaces over a field of positive characteristic like we have in characteristic $0$ where Hodge theory and Torelli Theorems were available. However, as there is no crystalline Torelli Theorem for non-supersingular K3 surfaces over a field of positive characteristic and the naive F-crystal (see Appendix) fails to be compatible with inner product,  the description in terms of F-crystals is not yet possible. Even though one has crystalline Torelli Theorem for supersingular K3 surfaces, it is essentially not providing any more information as there are no non-trivial Fourier-Mukai partners of a supersingular K3 surface. However, Lieblich-Olsson proved a derived Torelli Theorem using the Ogus Crystalline Torelli Theorem \cite{Ogus}, see \cite{LO2} Theorem 1.2.
\end{remark}

Let us already show here that height of a K3 surface is a derived invariant. This will allow us to stay within a subclass of K3 surfaces while checking derived equivalences. 

\begin{lemma} \label{derivedinvheight}
Height of a K3 surface $X$ over an algebraically closed field of characteristic $p > 3$ is a derived invariant. 
\end{lemma}

\begin{proof}
Recall that the height of a K3 surface $X$ is given by the dimension of the subspace $H^2_{crys}(X/K)_{[0,1)}$ of the F-isocrystal $H^2_{crys}(X/K)$. Now note that the Frobenius acts on the one dimensional isocrystals $H^0(X/K)(-1)$ and $H^4(X/K)(1)$ (Tate twisted) as multiplication by $p$ (see Appendix below for this computation). This implies that the slope of these F-isocrystals is exactly one. Thus, the F-isocrystal 
$$
H^*_{crys}(X/K) := H^0(X/K)(-1) \oplus H^2_{crys}(X/K) \oplus H^4_{crys}(X/K)(1)
$$ 
has the same subspace of slope of dimension strictly less than one as that of the F-isocrystal $H^2_{crys}(X/K)$, i.e.,   $H^*_{crys}(X/K)_{[0,1)} =H^2_{crys}(X/K)_{[0,1)}$. \\
Note that any derived equivalence of $X$ and $Y$ preserves the F-isocrystal $H^*_{crys}(-/K)$, i.e., if $\Phi_P: D^b(X) \simeq D^b(Y)$ is a derived equivalence of two K3 surfaces $X$ and $Y$, then the induced map on  the F-isocrystals
$$
\Phi_P^*: H^*_{crys}(X/K) \ra H^*_{crys}(Y/K)
$$  
is an isometry. Thus, for the height of $Y$ given by  $\text{dim}(H^2_{crys}(Y/K)_{[0,1)})$ we have
\begin{equation*}
\begin{split}
\text{dim}(H^2_{crys}(Y/K)_{[0,1)}) &= \text{dim}(H^*_{crys}(Y/K)_{[0,1)}) \\
&= \text{dim}(H^*_{crys}(X/K)_{[0,1)}) \\
&= \text{dim}(H^2_{crys}(X/K)_{[0,1)}) = \text{height of $X$}
\end{split}
\end{equation*}
Hence the result. 
\end{proof}

\begin{remark} 
\begin{enumerate}
\item In characteristic 0, there is no notion of height but in this case, the Brauer group itself is a derived invariant of a K3 surface, as $Br(X) \cong Hom(T(X), \mathbb{Q}/\Z)$, where $T(X)$ is the transcendental lattice. 
\item A related question would be: Is Picard rank a derived invariant for K3 surfaces? This is true trivially in characteristic $0$ (also see \cite{HuyLect} Corollary 16.2.8). The answer in characteristic $p$ is also yes, but the proof is not direct, it goes via characteristic $0$ for the finite height case using a Picard preserving lift as constructed by Leiblich-Maulik in \cite{LM} Corollary 4.2. And in the supersingular case it is preserved as any supersingular K3 does not have non-trivial Fourier-Mukai partners\footnote{The author would like to thank Vasudevan Srinivas for pointing it out to her that there is a characteristic independent proof of the fact that the Picard rank is a derived invariant for K3 surfaces: A derived equivalence between K3 surfaces yields an isomorphism between Grothendieck groups and hence also numerical Grothendieck groups, and for a surface, the rank of the numerical Grothendieck group is two more than the Picard rank.}. 
\item On the other hand, the Picard lattice is not a derived invariant in any characteristic, though it trivially remains invariant in the case of K3 surfaces which do not have non-trivial Fourier-Mukai partners.
\end{enumerate}
\end{remark}

\section{Derived Autoequivalences of K3 Surfaces in Positive Characteristic} \label{Derivedautoeqncharp}

In this section, we compare the deformation of an automorphism as a morphism and as a derived autoequivalence and show that for K3 surfaces these deformations are in one-to-one correspondence. Then we discuss Lieblich-Olsson's results on lifting derived autoequivalences.  Then we use these lifting results to prove results on the structure of the group of derived autoequivalences of a K3 surface of finite height over a field of positive characteristic.

\subsection{Obstruction to Lifting Derived Autoequivalences} \label{obstructiontoliftingauto}

We begin by recalling the classical result that for a variety the infinitesimal deformation of a closed sub-variety with a vanishing $H^1(X, \O_X)$ as a closed subscheme is determined by the deformation of its (pushforward of) structure sheaf as a coherent sheaf on $X \times X$. We then use this result to show that on a K3 surface we can lift an automorphism as a automorphism if and only if we can lift it as a perfect complex in the derived category. 

\begin{remark}
For a K3 surface this result can also be seen using \cite{LO} Proposition 7.1 and the $p$-adic criterion of lifting automorphisms on K3 surfaces \cite{EO} Remark 6.5.
\end{remark}

\begin{remark}
Note that in case of varieties that have $H^1(X, \O_X) \neq 0$, there are more ways of deforming the automorphism as a perfect complex, which in our case is just going to be a coherent sheaf (see below for a proof): For example, an elliptic curve or any higher genus curve.
\end{remark}

Let $X$ be a projective variety over an algebraically closed field $k$ of positive characteristic $p$, $W(k)$ its ring of Witt vectors and $\sigma: X \ra X$ an automorphism of $X$. We put the condition of characteristic $p > 3$ as at many places we may have denominators in factors of 2 and 3, like in the definition of Chern characters for K3 surfaces, and these will become invertible in $W(k)$ due to our assumption on the characteristic.

\begin{definition}
For any Artin local $W(k)$-algebra $A$ with residue field $k$, an \textbf{infinitesimal deformation of $X$ over $A$} is a proper and flat scheme $X_A$ over $A$ such that the following square is cartesian:
\begin{equation*}
\xymatrix{ X \ar[r] \ar[d] &X_A \ar[d] \\
\Spec(k) \ar[r] & \Spec(A). }
\end{equation*}
\end{definition}

\begin{remark}
In case $X$ is smooth, we ask $X_A$ to be smooth over $A$ as well. In this case, $X_A$ is automatically flat over $A$. 
\end{remark}

Consider the following two deformation functors:
\begin{equation}
\begin{split}
F_{aut}: & (\text{Artin local $W(k)$-algebras with residue field $k$})  \rightarrow (Sets) \\
& A \mapsto \{\text{Lifts of automorphism $\sigma$ to $A$} \},
\end{split}
\end{equation}
where by lifting of automorphism $\sigma$ over $A$ we mean that there exists an infinitesimal deformation $X_A$ of $X$ and an automorphism $\sigma_A: X_A \rightarrow X_A$ which reduces to $\sigma$, i.e., we have the following commutative diagram: 
$$
\xymatrix{
X_A  \ar[r]^{\sigma_A} &X_A\\
X \ar[u] \ar[r]^{\sigma} & X. \ar[u]}
$$
This is the \textbf{deformation functor of an automorphism as a morphism}. Now consider the \textbf{deformation functor of an automorphism as a coherent sheaf} defined as follows:
\begin{equation}
\begin{split}
F_{coh}: & (\text{Artin local $W(k)$-algebras with residue field $k$})  \rightarrow (Sets) \\
& A \mapsto \{\text{Deformations of $\mathcal{O}_{\Gamma(\sigma)}$ to $A$} \}/ iso,
\end{split}
\end{equation}
where by deformations of $\mathcal{O}_{\Gamma(\sigma)}$ to $A$ we mean that there exists an infinitesimal deformation $Y_A$ of $Y := X \times X $ over $A$ and a coherent sheaf $\mathcal{F}_A$, which is a deformation of the coherent sheaf $\mathcal{O}_{\Gamma(\sigma)}$ and $\mathcal{O}_{\Gamma(\sigma)}$ is considered as a coherent sheaf on $X \times X$ via the closed embedding $\Gamma(\sigma) \hookrightarrow X \times X$. Isomorphisms are defined in the obvious way. 

\begin{remark}
Note that there are more deformations of $X \times X$ than the ones of the shape $X_A \times_A X'_A$, where $X_A$ and $X'_A$ are deformations of $X$ over $A$. From now we make a choice of this deformation ($Y_A$) to be $X_A \times X_A$. Also see Theorem [\ref{Variationofdeformation}] and compare from Theorem [\ref{liftingkernels}] and Remark [\ref{Modulidirection}] below.
\end{remark}

Let $X$ be a smooth projective scheme over $k$ and for $A$ an Artin local $W(k)$-algebra assume that there exists an infinitesimal lift of $X$ to $X_A$\footnote{Note that such a lift may not always exist but for the case of K3 surfaces of finite height it does, see \cite{LM} Corollary 4.2 and Theorem \ref{Delignelifting}. However, for supersingular K3 surfaces, the lift does not exists over all Artin local rings but in some cases it does exist by Theorem \ref{Delignelifting}.}. Observe that there is a natural transformation $\eta: F_{aut} \rightarrow F_{coh}$ given by 
\begin{equation}
\begin{split}
\eta_A :  F_{aut}(A) & \longrightarrow F_{coh}(A)\\
 (\sigma_A: X_A \rightarrow X_A) & \mapsto \mathcal{O}_{\Gamma(\sigma_A)} / X_A \times X_A.
\end{split}
\end{equation}

\begin{theorem} \label{deformationcomparison}
The natural transformation $\eta:F_{aut} \rightarrow F_{coh} $ between the deformation functors is an isomorphism for varieties with $H^1(X, \O_X) = 0$. 
\end{theorem}

We provide an algebraic proof by constructing a deformation-obstruction long exact sequence connecting the two functors. The proof follows from the following more general proposition \ref{HartshorneEx19.1}, substituting $X \times X$ for $Y$ and taking the embedding $i$ to be the graph of the automorphism $\sigma$. To use proposition \ref{HartshorneEx19.1} we need the following lemma.

\begin{lemma}[Cf. \cite{HartshorneDT} Lemma 24.8] \label{24.8}  
To give an infinitesimal deformation of  an  automorphism $f: X \ra X$ over $X_A$ it is equivalent to give an infinitesimal deformation of the graph $\Gamma_f$ as a closed subscheme of $X \times X$.  
\end{lemma}

\begin{proof}
To any deformation $f_A$ of $f$ we associate its graph $\Gamma_{f_A}$, which gives a closed subscheme of $X_A\times X_A$. It is an infinitesimal deformation of $\Gamma_f$. Conversely, given a deformation $Z$ of $\Gamma_f$ over $A$, the projection $p_1: Z \hookrightarrow X_A \times_A X_A \ra X_A$ gives an isomorphism after tensoring with $k$. From flatness (see, for example, EGA IV, Corollary 17.9.5) of $Z$ over $A$ it follows that $p_1$ is an isomorphism, and so $Z$ is the graph of $f_A = p_2 \circ p_1^{-1}$.
\end{proof}

\begin{proposition} (Cf. \cite{HartshorneDT} Ex 19.1) \label{HartshorneEx19.1}
Let $i: X \hookrightarrow Y$ be a closed embedding with $X$ integral and projective scheme of finite type over $k$. Then there exists a long exact sequence 
\begin{equation}
0 \rightarrow H^0(\mathcal{N}_X) \rightarrow \text{Ext}^1_Y(\O_X, \O_X) \rightarrow H^1(\O_X) \rightarrow H^1(\mathcal{N}_X) \ra \text{Ext}^2_Y(\O_X, \O_X) \ra \ldots,
\end{equation}
where $\mathcal{N}_X$ is the normal bundle of $X$. 
\end{proposition}

\begin{proof}

Consider the short exact sequence given by the closed embedding $i$ 
\begin{equation}
0 \rightarrow I \rightarrow \O_Y \rightarrow i_*\O_X \rightarrow 0. 
\end{equation}

Apply the global Hom contravariant functor $\text{Hom}_Y(-, i_*\O_X)$ to the above short exact sequence and we get the following long exact sequence from \cite{HartshorneAG} III Proposition 6.4,

\begin{equation*}
\begin{split}
0 &\ra \text{Hom}_Y(i_*\O_X, i_*\O_X) \ra \text{Hom}_Y(\O_Y, i_*\O_X) \ra \text{Hom}_Y(I, i_*\O_X) \ra \\
& \text{Ext}^1_Y(i_*\O_X, i_*\O_X) \ra \text{Ext}^1_Y(\O_Y, i_*\O_X) \ra  \text{Ext}^1_Y(I, i_*\O_X) \ra \\
& \text{Ext}^2_Y(i_*\O_X, i_*\O_X) \ra \ldots .
\end{split}
\end{equation*}

Now note that we can make the following identifications
\begin{enumerate}
\item $\text{Hom}_Y(i_*\O_X, i_*\O_X) \cong k$ as $X$ is integral and projective. 
\item $\text{Hom}_Y(\O_Y, i_*\O_X) = H^0(\O_X) = k $ using \cite{HartshorneAG} III Propostion 6.3 (iii), Lemma 2.10 and the fact that $X$ is connected. 
\item As any injective endomorphism of a field is an automorphism, we can modify the long exact sequence as follows:
\begin{equation*}
0 \ra  \text{Hom}_Y(I, i_*\O_X) \ra \text{Ext}^1_Y(i_*\O_X, i_*\O_X) \ra \text{Ext}^1_Y(\O_Y, i_*\O_X) \ra \ldots .
\end{equation*}
  
\item $\text{Hom}_Y(I, i_*\O_X) \cong \text{Hom}_X(i^*I, \O_X)$ using ajunction formula on page 110 of \cite{HartshorneAG}. Moreover, using \cite{HartshorneAG} III, Proposition 6.9, we have 
$$
\text{Hom}_X(i^*I, \O_X) = \text{Hom}_X(\O_X, \mathcal{H}om_X(i^*I, \O_X)),
$$
and using the discussion in \cite{SP} Tag 01R1, we have $\mathcal{H}om_X(i^*I, \O_X) = \mathcal{N}_X$. Thus, putting this together with \cite{HartshorneAG} III Proposition 6.3 (iii) and Lemma 2.10, we get 
$$
\text{Hom}_Y(I, i_*\O_X) \cong H^0(\mathcal{N}_X).
$$   
\item Note that again using  \cite{HartshorneAG} III Proposition 6.3 (iii) and Lemma 2.10, we get 
$$
\text{Ext}^1_Y(\O_Y, i_*\O_X) \cong H^1(\O_X).
$$
\item Note that using the adjunction for Hom sheaves we have:
$$
i_*\mathcal{N}_X = i_*\mathcal{H}om_X(i^*I, \O_X) \cong \mathcal{H}om_Y(I, i_*\O_X).
$$
Thus, $H^1(\mathcal{N}_X) := H^1(X, \mathcal{N}_X) = H^1(Y, i_*\mathcal{N}_X)$ using \cite{HartshorneAG} III Lemma 2.10. To compute $ H^1(Y, i_*\mathcal{N}_X)$, we choose an injective resolution of $i_*\O_X$ as an $\O_Y$-module $0 \ra \O_X \ra \mathcal{J}^{\bullet}$. From \cite{tohoku} Proposition 4.1.3, we know that  $\mathcal{H}om_Y(I, \mathcal{J}^{i})$ are flasque sheaves and so we can compute the cohomology group using this flasque resolution. Hence,
\begin{equation*}
H^i = \frac{Ker(\text{Hom}_Y(I, \mathcal{J}^i) \ra \text{Hom}_Y(I, \mathcal{J}^{i+1}))}{ Im(\text{Hom}_Y(I, \mathcal{J}^{i-1}) \ra \text{Hom}_Y(I, \mathcal{J}^{i}))} = \text{Ext}^i_Y(I, i_*\O_X).
\end{equation*} 

\end{enumerate}

Thus, putting all of the above observations together, we get our required long exact sequence. 
\end{proof}

\begin{proof}[Proof of Theorem \ref{deformationcomparison}:] 
Note that the obstruction spaces for the functors $F_{aut}$ and $F_{coh}$ are $H^1(\mathcal{N}_X)$ and $\text{Ext}^2_Y(\O_X, \O_X)$ respectively.  See, for example, \cite{HartshorneDT} Theorem 6.2, Theorem 7.3, Exercise 7.4 and Lemma \ref{24.8} above. The same results give us the tangent spaces for the functors $F_{aut}$  and $F_{coh}$ and they are  $H^0(\mathcal{N}_X)$ and $\text{Ext}^1_Y(\O_X, \O_X)$.  Now using Proposition \ref{HartshorneEx19.1} along with our assumption of vanishing $H^1(X, \O_X)$ one has that the obstruction space of $F_{aut}$ is a subspace of the obstruction of $F_{coh}$ and this inclusion sends one obstruction class to the other one. Therefore, the obstruction to lifting the automorphism as a morphism vanishes if and only if the obstruction to lifting the automorphism as a sheaf vanishes. Moreover, the isomorphism of tangent spaces implies that the number of lifts in both cases is same.  
\end{proof}

This shows that for projective varieties with vanishing $H^1(X, \O_X)$, one doesn't have extra deformations of automorphisms as a sheaf. Note that we could still ask for deformations as a perfect complex but since the perfect complex we start with is a coherent sheaf any deformation of it as a perfect complex will also have only one non-zero coherent cohomology sheaf. Indeed, this follows from the fact that deformations cannot grow cohomology sheaves ,as if $F^{\bullet}_A$ is the deformation of $\O_X$ over $A$ such that $H^1(F^{\bullet}_A) \neq 0$ (to simplify our argument we are assuming $F^{\bullet}_A$ is bounded above at level 1, i.e., $F^i_A = 0 \ \forall i >1$), then we can replace this complex in the derived category by a complex like
 $$
\ldots \ra F^{-1}_A \ra Ker(F^0_A \ra F^1_A) \xrightarrow{0} H^1(F^{\bullet}_A) \ra 0.
$$
Then reducing to special fiber gives that $H^1(F_A^{\bullet}) \otimes_A k = 0$, but this will only happen if $H^i(F^{\bullet}_A) = 0$. Moreover, as we are in the derived category, we can show that the deformed perfect complex is then quasi isomorphic to a coherent sheaf. Indeed, the quotient map to the non-zero coherent cohomology sheaf provides the quasi-isomorphism. This shows that there are no extra deformations as a perfect complex as well.  Hence, an automorphism $\sigma$ on a projective variety $X$ with vanishing $H^1(X, \O_X)$  lifts if and only if the derived equivalence it induces, $\Phi_{\O_{\Gamma(\sigma)}}: D^b(X) \ra D^b(X)$, lifts as a Fourier-Mukai transform. 

\begin{remark}
\begin{enumerate} 
\item Note that we cannot claim that the derived equivalence lifts as a derived equivalence because in the relative setting when $\mathcal{X}$ is defined over $S$, where $S$ is a scheme not equal to $\Spec(k)$, one does not have the Orlov Representability Theorem  \ref{ORT} and therefore, a priori, one cannot say that every derived equivalence comes from a Fourier-Mukai Transform. Thus, a priori, we can possibly lift more things as a derived equivalence. 
\item If we use the infinity category of perfect complexes on X, denoted by $Perf(X)$, in place of the derived category on X, we can then say that an automorphism $\sigma$ on a projective variety $X$ with vanishing $H^1(X, \O_X)$  lifts if and only if the equivalence  $ \Phi_{\O_{\Gamma(\sigma)}}: Perf(X) \ra Perf(X)$ lifts as an autoequivalence of the infinity category of perfect complexes, as in this case we have a representability Theorem \cite{BZ-F-N}.
\end{enumerate}
\end{remark}

Now we state the two theorems proved by Lieblich-Olsson which give a criteria to lifting perfect complexes.  

\begin{theorem}[\cite{LO} Theorem 6.3] \label{Infinitesimallifting}  \label{liftingkernels}
Let $X$ and $Y$ be two K3 surfaces over an algebraically closed field $k$, and $P \in D^b( Y \times X)$ be a perfect complex inducing an equivalence $\Phi: D^b(Y) \ra D^b(X)$ on the derived categories. Assume that the induced map on cohomology (see below) satisfies:
\begin{enumerate}
\item $\Phi(1, 0, 0) = (1, 0, 0)$, 
\item the induced isometry $\kappa: Pic(Y) \rightarrow Pic(X)$ sends $C_Y$, the ample cone of Y, isomorphically to either $C_X$ or $-C_X$, the $(-)$ample cone of $X$.   
\end{enumerate}  
Then there exists an isomorphism of infinitesimal deformation functors $\delta: Def_X \rightarrow Def_Y$ such that 
\begin{enumerate}
\item $\delta^{-1}(Def_{(Y,L)}) = Def_{(X, \Phi(L))}$;
\item for each augmented Artinian $W$-algebra $W \rightarrow A$ and each $(X_A \rightarrow A) \in Def_X(A)$, there is an object $P_A \in D^b(  \delta(X_A) \times_A X_A)$ reducing to $P$ on $Y \times X$. 
\end{enumerate}
\end{theorem}

\begin{theorem} [\cite{LO}, Theorem 7.1] \label{Hodgefil} 
Let $k$ be a perfect field of characteristic $p > 0$, $W$ be the ring of  Witt vectors of $k$, and $K$ be the field of fractions of $W$. Fix K3 surfaces $X$ and $Y$ over $k$ with lifts $X_W/W$ and $Y_W/W$. These lifts induce corresponding Hodge filtrations via de Rham cohomology on the crystalline cohomology of the special fibers. Denote  by $F^1_{Hdg}(X) \subset H^2(X/K) \subset H^*(X/K)$ and  $F^1_{Hdg}(Y) \subset H^2(Y/K) \subset H^*(Y/K)$ (similarly for $F^2_{Hdg}(-)$), where $ H^*(X/K)$ and $ H^*(Y/K)$ are the corresponding Mukai F-isocrystals. Suppose that $P \in D^b(X \times Y)$ is a kernel whose associated functor $\Phi: D^b(X) \ra D^b(Y)$ is fully faithful. If 
$$
\Phi:  H^*(X/K) \ra  H^*(Y/K)
$$
sends $F^1_{Hdg}(X)$ to $F^1_{Hdg}(Y)$ and $F^2_{Hdg}(X)$ to $F^2_{Hdg}(Y)$, then $P$ lifts to a perfect complex $P_W \in D^b(X_W \times_W Y_W)$.  
\end{theorem}

We refer the reader to \cite{LO} for the proof.

\begin{remark}
Note that however, this is not true infinitesimally. We have the same counterexamples as in the case of infinitesimal integral variational Hodge conjecture: take a line bundle such that $\mathcal{L}^{\otimes p} \neq \O_X$, then we have the Chern character of $\mathcal{L}^{\otimes p}$ is $0$ as $ p. ch(\mathcal{L}) = 0$, so it lies in the correct Hodge level, but it need not lift. For example: see \cite{BO} Lemma 3.10.
\end{remark}

\begin{remark}
Note that the lifted kernel also induces an equivalence. Indeed, for a K3 surface fully faithful Fourier-Mukai functor of derived categories is an equivalence (see \cite{HuyFM} Proposition 7.6) and so we can also lift the Fourier-Mukai kernel of the inverse equivalence. Then the composition of the equivalence we started with and its inverse will give us a lift of the identity as an derived autoequivalence. But using the fact that the $\text{Ext}^1_{X \times X} (P, P) = 0$  (see Lemma \ref{4.34}) for any kernel inducing an equivalence, we get that the lift of the identity is unique and is the identity itself. Thus, the lifted Fourier-Mukai functor is an equivalence. 
\end{remark}

\begin{corollary}
Take $P$ to be $\O_{\Gamma(\sigma)}$, where $\sigma: X \ra X$ is an automorphism of a K3 surface $X$ over $k$. Then $P$ lifts to an autoequivalence of $D^b(X_W)$ if and only if $\sigma$ lifts to an automorphism of $X_W$ if and only if $P$ preserves the Hodge filtration. 
\end{corollary} 

However, we see that we can still lift it as an isomorphism as follows:
\begin{theorem}[Weak Lifting of Automorphisms]
Let $\sigma: X \rightarrow X$ be an automorphism of a K3 surface $X$ defined over an algebraically closed field $k$ of characteristic $p$. There exists a smooth projective model $X_R /R$, where $R$ is a discrete valuation ring that is a finite extension of $W(k)$, with $X_K$ its generic fiber such that there is a $P_R$, a perfect complex in $D^b (X_R \times Y_R)$, reducing to $\mathcal{O}_{\Gamma(\sigma)}$ on $X \times X$, where $Y_R$ is another smooth projective model abstractly isomorphic to $X_R$ (see Remark [\ref{iso}] below).
\end{theorem}

\begin{proof}
We divide the proof into 3 steps:
\begin{enumerate}

\item Lifting Kernels Infinitesimally:
 Note that $\Phi_{\mathcal{O}_{\Gamma(\sigma)}}$ is a strongly filtered derived equivalence, i.e., 
\begin{equation*}
\Phi_{\mathcal{O}_{\Gamma(\sigma)}}^* = \sigma^*: H^i_{crys}(X/W) \xrightarrow{\sim} H^i_{crys}(X/W)
\end{equation*} 
is an isomorphism which preserves the gradation of crystalline cohomology. Choose a projective lift of $X$ to characteristic zero along with a lift of $H_X$. It always exists as proved by Deligne \cite{Deligne}, i.e., a projective lift $(X_V, H_{X_V})$ of $(X, H_X)$ over $V$ a discrete valuation ring, which is a finite extension of $W(k)$, the Witt ring over $k$. Let $V_n := V/ \mathfrak{m}^n$ for $n \geq 1$, $\mathfrak{m}$ the maximal ideal of V  and let $K$ denote the fraction field of $V$. Then, for each $n$, using the lifting criterion above, there exists a polarized lift $(X'_n, H_{X'_n})$ over $V_n$ and a complex $P_n \in D_{Perf}(X_n \times X'_n)$ lifting $\mathcal{O}_{\Gamma(\sigma)}$.

\item Applying the Grothendieck Existence Theorem for perfect complexes:  
By the classical Grothendieck Existence Ttheorem \cite{HartshorneAG}, II.9.6, the polarized formal scheme $(\varprojlim X'_n, \varprojlim H_{X'_n})$ is algebraizable. So, there exists a projective lift $(X', H_{X'})$ over $V$ that is the formal completion of $(X'_n, H_{X'_n})$. Now using the Grothendieck Existence Theorem for perfect complexes (see \cite{L} Proposition 3.6.1) the formal limit of $(P_n)$ is algebraizable and gives a complex $P_V \in D_{Perf}(X_V \times X'_{V})$. In particular, $P_V$ lifts $\mathcal{O}_{\Gamma(\sigma)}$ and using Nakayama's lemma, $P_V$ induces an equivalence. 

\item Now apply the global Torelli Theorem to show that the two models are isomorphic:
For any field extension $K'$ over $K$, the generic fiber complex $P_{K'} \in D^b(X_{K'} \times X'_{K'})$ induces a Fourier-Mukai equivalence $\Phi_{P_{K'}}:D(X_{K'}) \rightarrow D(X'_{K'})$. Using Bertholet-Ogus isomorphisms \cite{BO}, we see that $\Phi_{K'}$ preserves the gradation on de Rham cohomology of $X_{K'}$. Fix an embedding of $K' \hookrightarrow \mathbb{C}$ gives us a filtered Fourier Mukai equivalence 
\begin{equation*} 
\Phi_{P_{\mathbb{C}}}: D^b(X_{K'} \times \mathbb{C}) \rightarrow D^b(X'_{K'} \times \mathbb{C}),
 \end{equation*}
which in turn induces an Hodge isometry of integral lattices:
\begin{equation*}
H^2(X_{K'} \times \mathbb{C}, \mathbb{Z}) \xrightarrow{\sim} H^2(X'_{K'} \times \mathbb{C}, \mathbb{Z}),
\end{equation*}
using Theorem \ref{DerivedHodge} and the fact that a filtered equivalence preserves the grading. 
This implies that $X_{K'} \times \mathbb{C} \cong X'_{K'} \times \mathbb{C}$, which after taking a finite extension $V'$ of $V$ gives that the generic fiber are isomorphic  $X_{K'} \cong X'_{K'}$ (we abuse notation to still denote the fraction field of $V'$  by $K'$). And since the polarization was lifted along, this gives actually a map of polarized K3 surfaces denoted by $f_{K'}: (X_{K'}, H_{X_{K'}}) \xrightarrow{\sim} (X'_{K'}, H_{X'_{K'}})$.\\
\end{enumerate}
Now we can conclude that the the generic fibers are isomorphic as well by forgetting the polarization. So now we need to show that the models are isomorphic, i.e., $X_{V'} \cong X'_{V'}$, which will follow from the following proposition.
\end{proof}

\begin{proposition}[Matsusaka-Mumford, \cite{MM}] \label{MM}
Let $X_R$ and $Y_R$ be two varieties over a discrete valuation ring $R$ with residue field $k$,  $X_K$ and $Y_K$ be their generic fibers defined over $K$, the fraction field of $R$, and  the special fibers $X_k$ and $Y_k$ be non-singular varieties. Assume that $X_K, Y_K, X_k, Y_k$ are underlying varieties of polarized varieties, $(X_K, H_{X_K}), (Y_K, H_{Y_K}), (X_k, H_{X_k}), (Y_k, H_{Y_k})$ and that the specialization map extends to the polarized varieties. Then, for $Y_k$ not ruled, if there is an isomorphism $f_K: (X_K, H_{X_K}) \rightarrow (Y_K, H_{Y_K})$, $f_K$ can be extended to an isomorphism $f_R$ of $X_R$ and $Y_R$. Moreover, the graph of $f_K$ specializes to the graph of an isomorphism $f_k$ between $(X_k, H_{X_k})$ and $(Y_k, H_{Y_k})$. 
\end{proposition}

\begin{remark} \label{iso}
Note that even though the generic fibers are isomorphic which indeed implies that the models are abstractly isomorphic (via the Matsusaka-Mumford Theorem) but not as models of the special fiber as the isomorphism will not be the identity on the special fiber, just for the simple reason that we started with different polarizations on the special fibers. 
\end{remark}

\begin{remark}
This dependence on the choice of the lift $X_A$ of $X$  and the ability to find another lift $Y_A$ can be seen as a reformulation of the formula stated below:
\begin{theorem}[\cite{HuyRT} Theorem on page 2]  \label{Variationofdeformation}
Let $P_0$ be a perfect complex on a separated noetherian scheme $X_0$ and let $i: X_0 \hookrightarrow X$  be a closed embedding defined by an ideal $I$ of square zero. Assume that $X$ can be embedded into a smooth ambient space $A$ (for example if $X$ is quasi-projective). Then there exists a perfect complex $P$ on $X$ such that the derived pullback $i^*P$ is quasi-isomorphic  to $P_0$ if and only if 
\begin{equation*}
0 = (id_{P_0} \otimes \kappa(X_0/X)) \circ A(P_0) \in \text{Ext}^2_{X_0}(P_0, P_0 \otimes I),
\end{equation*}
where $A(P_0)$ is the (truncated) Atiyah class and $\kappa(X_0/X)$ is the (truncated) Kodaira-Spencer class.   
\end{theorem}
\end{remark}

\begin{remark} \label{Modulidirection} 
The above results can be rephrased to say that in the moduli space of lifts of $X \times X$ we cannot always deform the automorphism in the direction of $X_A \times X_A$ but can do so always in the direction of some $X_A \times Y_A$, where $X_A$ and the automorphism determine $Y_A$ uniquely.  
\end{remark}

Next, we discuss the structure of the derived autoequivalence group of a K3 surface of finite height. 

\subsection{The Cone Inversion Map}

Let $X$ be a K3 surface over $k$ of finite height with $char(k) = p > 3$.
\begin{definition}
The \textbf{positive cone} $ \mathcal{C}_X \subset NS(X)_{\R}$ is the connected component of the set $\{\alpha \in NS(X)| (\alpha)^2 > 0\}$ that contains one ample class (or equivalently, all of them).  
\end{definition}

\begin{definition}[Cone Inversion map] Let $\mathcal{C}_X$ be the positive cone, the \textbf{cone inversion map} on the cohomology is the map that sends the positive cone $\mathcal{C}_X$ to $-\mathcal{C}_X$. 
\end{definition}
Explicitly, in characteristic 0, we define the map to be $(-id_{H^2}) \oplus id_{H^0 \oplus H^4}: \tilde{H}^*(X, \Z) \ra  \tilde{H}^*(X, \Z)$, where $\tilde{H}^*(X, \Z)$ is the Mukai lattice (\cite{HuyFM}, Section 10.1). Note that the cone inversion map is a Hodge isometry. In characteristic $p >3$, we define the map to be $(-id_{H^2}) \oplus id_{H^0 \oplus H^4}: H^*_{crys}(X/K) \ra  H^*_{crys}(X/K)$, where $H^*_{crys}(X/K)$ is the Mukai F-isocrystal (see appendix below). Note that the cone inversion map preserves the Hodge Filtration on $H^2_{crys}(X/K)$. 

(In characteristic 0, the following proposition is proved in \cite{Huy} with the Mukai F-crystal replaced with Mukai lattice.).  

\begin{theorem} \label{orientation preserving}
The image of $Aut(D^b(X))$ in $Aut(\mathcal{H}^*_{crys}(X/K))$ has index at least 2, where $\mathcal{H}^*_{crys}(X/K)$ is the Mukai F-isocrystal.
\end{theorem}

We prove the above proposition by showing that the cone inversion map on the cohomology does not come from any derived auto-equivalence. The proof is done by contradiction, we assume that such an auto-equivalence exists, then lift the kernel of the derived auto-equivalence to char $0$, and then we use the results of \cite{Huy}, to get a contradiction that this does not happen.

Recall that we have the following diagram of descend to cohomology of a Fourier-Mukai transform $\Phi_P$, for $P \in D^b(X \times Y)$:
\begin{center}
$\begin{CD}
 D^b(X)@>\text{$\mathcal{E} \mapsto \mathbb{R}p_{Y*}(p_X^* \mathcal{E}) \otimes^{\mathbb{L}} P) $}>>D^b(X)\\
@VVch( \ )V @VVch( \ )V\\
CH^*(X) @>>>CH^*(X)\\
@VV{\circ \sqrt{td_{X}}}V @VV\circ \sqrt{td_{Y}}V\\
 H^*(X)@>\text{$\alpha \mapsto p_{Y*}\big(( p_X^*\alpha)\cdot ch(P) \sqrt{td_{X\times Y}} \ \big) $}>> H^*(X),
\end{CD}$\\
\end{center} 
where $ch(-)$ is the Chern character and $td_{-}$ is the Todd genus. 
Before proving the above theorem we state the following lemma which will be required for the proof.

\begin{lemma}[\cite{LM}, Corollary 4.2] \label{lmlift}
Any K3 surface of finite height over a perfect field $k$ is the closed fiber of a smooth projective relative K3 surfaces $X_W \ra \Spec(W(k))$ such that the restriction map $Pic(X_W) \ra Pic(X)$ is an isomorphism. Moreover, it preserves the positive cone.
\end{lemma}

\begin{proof}[Proof of Theorem \ref{orientation preserving}]
Assume that the cone inversion map is induced by a derived auto-equivalence. Then using Orlov's representability Theorem (\cite{Orlov1}, \cite{Orlov2}), we know that this derived auto-equivalence is  a Fourier-Mukai transform and we denote the kernel of the transform by $\mathcal{E}$. Since $\mathcal{E}$ induces the cone inversion map and this map preserves the Hodge filtration on the crystalline cohomology, using Theorem \ref{Hodgefil}, we know that we can lift the perfect complex $\mathcal{E}$ to a perfect complex $\mathcal{E}_W$ in $D^b(X_W \times X_W)$, where $X_W$ is the lift of $X$ as in Lemma \ref{lmlift}. Note that the lifted complex also induces a derived equivalence. Indeed, using Nakayama's lemma we see that the adjunction maps $\Delta_*\O_{X_W} \ra \mathcal{E}_W \circ \mathcal{E}_W^{\vee}$ and  $\mathcal{E}_W \circ \mathcal{E}_W^{\vee} \ra \Delta_*\O_{Y_W}$  are quasi-isomorphisms. Moreover, since we have $H^*_{crys}(X/W) \cong H^*_{DR}(X_W/W)$, we know that the lifted complex induces again the cone inversion map on the cohomology. It also follows that for any field extension $K'/K$, the generic fiber complex $\mathcal{E}_{K'} \in D^b(X_{K'} \times_{K'} X_{K'})$ induces a Fourier Mukai equivalence $\Phi: D^b(X_{K'}) \ra D^b(X_{K'})$. Choosing an embedding $K \hookrightarrow \C$ (see our conventions [\ref{convention}]) yields a Fourier-Mukai  equivalence $ D^b(X_{K} \otimes \C) \ra D^b(X_{K} \otimes \C)$ which induces the cone inversion map on $\tilde{H}^*(X, \Z)$. This is a contradiction as in characteristic zero this does not happen, see \cite{Huy} for a proof. 
\end{proof}

We now make an interesting observation about the kernel of the map:

\begin{corollary}\label{Kernellifts}
Let $X$ be a K3 surface over $k$, an algebraically closed field of positive characteristic. Then the kernel of the natural map 
\begin{equation*}
0 \ra Ker \ra Aut(D^b(X)) \ra Aut(H^*_{crys}(X/K))
\end{equation*}
lifts. More precisely, assume that $X_V$ be a lift of $X$ over $V$, a mixed characteristic discrete valuation ring with residue field $k$, then every derived autoequivalence in the kernel of the map above lifts as an autoequivalence of the derived category of $X_V$. 
\end{corollary} 

\begin{proof}
This is clear as any autoequivalence in the kernel induces the identity automorphism on the cohomology which is bound to respect every Hodge filtration on the F-isocrystal and then we use Theorem \ref{Hodgefil}.
\end{proof}

This allows us to give at least an upper bound on the kernel as follows:
Let $X$ be a K3 surface over an algebraically closed field of characteristic $p >2$. Choose a lift of $X$, denoted as $X_R$, such that the Picard rank of the geometric generic fiber is $1$. There always exists such a lift as shown by Esnault-Oguiso.

\begin{theorem}[Esnault-Oguiso \cite{EO}, Theorem 4.1]
Let $X$ be a K3 surface defined over an algebraically closed field $k$ of characteristic $p >0$, where $p >2$ if $X$ is supersingular. Then there is a discrete valuation ring $R$, finite over the ring of Witt vectors $W(k)$, together with a projective model $X_R \ra \Spec(R)$, such that the Picard rank of $X_{\bar{K}}$  is $1$, where $K$ is the fraction field of $W(k)$ and $\bar{K} \supset K$ is an algebraic closure.  
\end{theorem} 

Let $\Phi_P: D^b(X) \ra D^b(X)$ be a Fourier-Mukai autoequivalence induced by $P \in D^b(X \times X)$ that belong to the kernel of the natural map 
$$
Aut(D^b(X)) \ra Aut(H^*_{crys}(X/K)).
$$ 
We will denote the kernel of this map as $Ker_X$. Now using the following lemma we see that the set of infinitesimal deformations of the kernel $P$ is a singleton set, which in turn implies that the lift of $P$ to $X_R \times X_R$  (this was just the corollary \ref{Kernellifts}) is unique. 

\begin{lemma} \label{4.34}
Let $X$ and $Y$ be K3 surfaces over an algebraically closed field $k$ and let $P \in D(X \times Y)$ be a complex defining the Fourier-Mukai equivalence $\Phi_P: D(X) \ra D(Y)$. Then $\rm{Ext}^1_{X \times Y}(P, P) = 0$.  
\end{lemma}

\begin{proof} 
See \cite{LO2}, Lemma 3.7 (ii).
\end{proof}

Next, note that the fiber of the lift of $P$ over the geometric generic point of $R$, denoted as $P_{\bar{K}}$, also belongs to the kernel of the natural map (again base changed to $\C$ using the embedding $ \bar{K} \subset \C$)
$$
Aut(D^b(X_{\C})) \ra O_{Hdg}(\tilde{H}^*(X_{\C}, \Z)),
$$
denoted as $Ker_{X_{\C}}$.  Indeed, this follows from the base change on cohomology and Berthelot-Ogus's isomorphism \cite{BO}. Let us assume that $\Phi_{P_{\C}}$ does not induces the identity on the singular cohomology of $X_{\C}$ and hence, using the following natural commutative diagram
$$
\xymatrix{ H^*(X_{\C}, \C) \ar[r] \ar[d]^{\cong} &H^*(X_{\C}, \C) \ar[d]^{\cong} \\
H^*_{DR}(X_{\C}) \ar[r] &H^*_{DR}(X_{\C}), }
$$
$\Phi_{P_{\C}}$ also does not induces the identity on the de Rham cohomology of $X_{\C}$. As the autoequivalence $\Phi_{P_{\C}}$ is just the base change of $\Phi_{P_{\bar{K}}}$ we see that the map induced by $\Phi_{P_{\bar{K}}}$ on the de Rham cohomology of $X_{\bar{K}}$ is not the identity. Now again $\Phi_{P_{\bar{K}}}$ comes via base change from $\Phi_{P_K}$ so it is not the identity on de Rham cohomology of $X_K$, now using the Berthelot-Ogus's isomorphism it does not induce the identity on the crystalline cohomology of $X$ but this is not possible as it is a lift of an autoequivalence which induces the identity on the crystalline cohomology.

This gives us the following injective map 
\begin{equation*}
\begin{split}
Ker_X &\hookrightarrow Ker_{X_{\C}} \\
\Phi_P &\mapsto \Phi_{P_{\C}}.
\end{split}
\end{equation*}

Now, using the Picard rank 1 lift, we see that $Ker_X$ is a subgroup of the kernel, $Ker_{X_{\C}}$. And this kernel has been described in \cite{BB} Theorem 1.4. Thus, we have shown that
 
\begin{proposition} \label{boundonKernel}
Let $X$ be a K3 surface over $k$, an algebraically closed field of characteristic $p > 3$, and $X_R \ra \Spec(R)$ be a Picard rank one lift of X with $X_{\C}$ the base change to $\C$ of the geometric generic fiber of $X_R$. Here, $R$ is mixed characteristic discrete valuation ring with residue field $k$. Then $Ker_{X} \subset Ker_{X_{\C}}$. 
\end{proposition}

\section{Counting Fourier-Mukai Partners in Positive Characteristic} \label{CountingFMP}

In this last section, we count the number of Fourier-Mukai partners of an ordinary K3 surface, in terms of the Fourier-Mukai partners of the geometric generic fiber of its canonical lift. Moreover, we prove that any automorphism of ordinary K3 surfaces lifts to its canonical lift. We start with comparing the Fourier-Mukai partners of a K3 surface over a field of positive characteristic with that of the geometric generic fiber of its lift to characteristic zero. Then we restrict to ordinary K3 surfaces and give a few consequences to lifting automorphisms of ordinary K3 surfaces. Moreover, we give a sufficient condition on derived autoequivalences of an ordinary K3 surface so that they lift to the canonical lift.  Lastly, we show that the class number counting formula (compare from Theorem \ref{classnocount}) also holds for K3 surfaces over a characteristic $p$ field.

Let $X$ (resp. $Y$) be a regular proper scheme with $D^b(X)$ (resp. $D^b(Y)$) its bounded derived category. Recall that we say that $Y$ is a \textbf{Fourier-Mukai partner} of $X$ if there exists a perfect complex $\mathcal{P} \in D^b(X \times Y)$ such that the following map is an equivalence of derived categories:
\begin{equation}
\begin{split}
\Phi_P:  D^b(X) & \xrightarrow{\cong} D^b(Y) \\
		\mathcal{Q} & \mapsto \mathbb{R}p_{Y*} ((p_X^* \mathcal{Q}) \otimes^{\mathbb{L}} \mathcal{P}), 
\end{split}
\end{equation}
where $p_X$ (resp. $p_Y$) is the projection from $X \times Y$ to $X$ (resp. $Y$).

We want to count the number of Fourier-Mukai partners of a K3 surface in positive characteristic. We will do this by lifting the K3 surface to characteristic 0 and then counting the Fourier-Mukai partners of the geometric generic fibers. For this we will show that the specialization map for Fourier-Mukai partners defined below is injective and surjective:
\begin{equation} \label{spzmap}
\begin{split}  \{ \text{FM partners of $X_{\bar{K}}$} \}  \ra &\{ \text{FM partners of X} \} \\
M_{X_{\bar{K}}}(v) \mapsto &M_X(v).
\end{split}
\end{equation}
Here, $X$ is a K3 surface of finite height over $k$ an algebraically closed field of characteristic $p > 3$,  $X_{\bar{K}}$ is the geometric generic fiber of $X_W$, which is a Picard preserving lift of $X$, and $M_X(v)$ (resp.\ $M_{X_{\bar{K}}}(v)$, $M_{X_W}(v)$) is the (fine) moduli space of stable sheaves with Mukai vector $v$ on $X$ (resp. $X_{\bar{K}}$, $X_W$). Note that from now on we will fix one such lift of $X$. Such a lift always exists by Lemma \ref{lmlift} for K3 surfaces of finite height. On the other hand, Theorem \ref{LOmainthm} shows that supersingular K3 surfaces have no nontrivial Fourier-Mukai partners, so from now we restrict to the case of K3 surfaces of finite height. 

To show that the map (\ref{spzmap}) is well defined, we need the following lemma: 

\begin{lemma}[(Potentially) Good reduction] (\cite{LO2} Theorem 5.3)
Let $V$ be a discrete valuation ring with a fraction field $K$, a field of characteristic $0$, and residue field $k$ of characteristic $p$ such that there is a K3 surface $X_K$ over $K$ with good reduction, then all the Fourier-Mukai partners of $X_{\bar{K}}$ have good reduction possibly after a finite extension of $K$. 
\end{lemma}

Thus for any Fourier-Mukai partner of $X_{\bar{K}}$ which is of the form $M_{X_{\bar{K}}}(v)$ is a geometric generic fiber of $M_{X_V}(v) / V$, where $V$ is a finite (algebraic) extension of $W(k)$. Note that the residue field of $V$ is still $k$ as $k$ is algebraically closed. Now using functoriality of the moduli functor we note that the special fiber of $M_{X_V}(v)$ is $M_X(v)$. This is a Fourier-Mukai partner of $X$ (see, for example, \ref{LOmainthm}).  Thus, the map (\ref{spzmap}) is well-defined.

\begin{proposition} [Lieblich-Olsson \cite{LO}] \label{Prop5.2} \label{surjectivity}
The specialization map (\ref{spzmap}) above is surjective.
\end{proposition}

\begin{proof}
From \cite{LO} Theorem 3.16, note that all Fourier-Mukai partners of $X$ are of the form $M_X(v)$. Moreover, one can always assume $v$ to be of the form $(r, l, s)$ where $l$ is the Chern class of a line bundle and $r$ is prime to $p$ (see \cite{LO}, Lemma 8.1). (Note that we take the Mukai vector here in the respective Chow groups rather than cohomology groups).   Then since we have chosen our lift $X_W$ of $X$ to be Picard preserving, we can also lift the Mukai vector to $(r_W, l_W, s_W)$, again denoted by $v$,  and this gives a FM partner of $X_W$, namely $M_{X_W}(v)$, and taking the geometric generic fiber of it gives a Fourier-Mukai partner of $X_{\bar{K}}$. 
\end{proof}

\begin{remark}
Note that the $Pic(X_{\bar{K}}) \cong Pic(X)$, i.e., the specialization map is an isomorphism. This is essentially due to the fact that $k$ is algebraically closed and every line bundle on $X$ lifts uniquely to $X_W$ as $Ext^1(L, L) = H^1(X, \O_X) = 0$ for $L \in Pic(X)$, under which the set of infinitesimal deformations of the line bundle $L$ is a torsor. 
\end{remark}

\begin{remark}
Note that the argument above already implies that the number of Fourier-Mukai partners of a K3 surface over an algebraically closed field of characteristic $p > 3$ is finite. This argument was given by Lieblich-Olsson in \cite{LO}.
\end{remark}

\textbf{Injectivity:}We need to show that if $M_X(v) \cong X$, then $M_{X_W}(v) \cong X_W$. For this statement we will restrict to the case of ordinary K3 surfaces. We recall some results about ordinary K3 surfaces and their canonical lifts as proved by Nygaard in \cite{Nygaardtate} and \cite{Nygaardtorelli}, and by Deligne-Illusie in \cite{Deligne-Illusie}. 

\begin{definition}[Ordinary K3 surface]
A K3 surface $X$ over a perfect field $k$ of positive characteristic is called \textbf{ordinary} if the height of $X$ is 1. 
\end{definition}

\begin{proposition} The following are equivalent:
\begin{enumerate}
\item $X$ is an ordinary K3 surface,
\item The height of formal Brauer group is 1,
\item The Frobenius $F: H^2(X, \O_X) \ra H^2(X, \O_X)$ is bijective. 
\end{enumerate}
\end{proposition}

We refer to \cite{Nygaardtate} Lemma 1.3 for a proof of this proposition. 

Let $A$ be an Artin local ring with residue field $k$ and let $X_A/ A$ be a lifting of the ordinary K3 surface $X/ k$. In \cite{AM}  Artin-Mazur showed that the enlarged Brauer group $\Psi_{X_A}$ defines a $p$-divisible group on $\Spec(A$) lifting $\Psi_X / k$. 

\begin{theorem}[Nygaard \cite{Nygaardtate}, Theorem 1.3]
Let $X/k$ be an ordinary K3 surface. The map
\begin{equation*}
\{ \text{Iso. classes of liftings $X_A/A$} \} \ra \{ \text{Iso. classes of liftings G/A} \}
\end{equation*}
defined by 
\begin{equation*}
X_A/A \mapsto \Psi_{X_A}/A
\end{equation*}
is a functorial isomorphism. 
\end{theorem}

Recall that the enlarged Brauer group of a K3 surface fits in the following exact sequence (\cite{AM} Proposition IV.1.8):
\begin{equation}
0 \ra \Psi_X^0 (= \hat{Br_X}) \ra \Psi_X \ra \Psi^{\acute{e}t} \ra 0.
\end{equation}
 
As the height one formal groups are rigid, there is a unique lifting $G^0_A$ of $\Psi_X^0$ to $A$. Similarly, the \'etale groups are rigid as well, so there is a unique lift $G^{\acute{e}t}_A$ of $\Psi_X^{\acute{e}t}$ to $A$. This implies that if $G$ is any lifting of $\Psi_X$ to $A$, then we have an extension
$$
0 \ra G^0_A \ra G \ra G^{\acute{e}t}_A \ra 0
$$
lifting the extension
$$
0 \ra \Psi_X^0 \ra \Psi_X \ra \Psi_X^{\acute{e}t} \ra 0.
$$
Therefore, the trivial extension $G = G^0_A \times G_A^{\acute{e}t}$ defines a unique lift $X_{can,A}/A$  of $X/k$ such that $\Psi_{X_{can,A}} = G^0_A \times G^{\acute{e}t}_A$. Take $A = W_n$ and $X_n = X_{can, W_n}$, then we get a proper flat formal scheme $\{X_n\}/ Spf W$.

\begin{theorem}[Definition of \textbf{Canonical Lift}]
The formal scheme $\{X_n\} / Spf W$ is algebraizable and defines a K3 surface $X_{can}/ \Spec(W)$. 
\end{theorem}
This theorem was proved by Nygaard in \cite{Nygaardtate}, Proposition 1.6. 

One of the nice properties of the canonical lift is that it is a Picard lattice preserving lift.

\begin{proposition}[Nygaard, \cite{Nygaardtate}, Proposition 1.8]
The canonical lift $X_{can}$ has the property that any line bundle on $X$ lifts uniquely to $X_{can}$.
\end{proposition}

Next, we state a criteria for a lifting of an ordinary K3 surface to come from the canonical lift. This is the criteria that we will be using to determine that our lift is canonical. 

\begin{theorem}[Taelman \cite{Taelman} Theorem C] \label{lenny}
Let $\O_K$ be a discrete valuation ring with perfect residue field $k$ of characteristic $p$ and fraction field $K$ of characteristic $0$. Let $X_{\O_K}$ be a projective $K3$ surface over $\O_K$ with $X_{\bar{K}}$ the geometric generic fiber and assume that $X := X_{\O_K} \otimes k$, the special fiber, is an ordinary K3 surface. Then the following are equivalent:
\begin{enumerate}
\item $X_{\O_K}$ is the base change from $W(k)$ to $\O_K$ of the canonical lift of $X$,
\item $H^2_{et}(X_{\bar{K}}, \Z_p) \cong H^0 \oplus H^1(-1) \oplus H^2(-2)$ with $H^i$ unramified $\Z_p[Gal_K]$-modules, free of rank $1, \ 20, \ 1$ over $\Z_p$ respectively. 
\end{enumerate} 
Here, the $(-1)$ and $(-2)$ denote Tate twists. 
\end{theorem}

We now prove that the automorphisms of an ordinary K3 surface lift always to characteristic zero. 

\begin{theorem} \label{liftingauto}
Every isomorphism $\vp: X \ra Y$ of ordinary K3 surfaces over an algebraically closed field of characteristic $p$ lifts to an isomorphism of the canonical lift of the ordinary K3's $\vp_W: X_{can} \ra Y_{can}$. In particular, every automorphism of $X$ lifts to an automorphism of $X_{can}$.    
\end{theorem}

\begin{remark}
Note that the above statement is stronger than the tautological statement: If $X$ and $X'$ are two isomorphic ordinary $K3$ surfaces over a perfect field $k$, then their canonical lifts are isomorphic. 
\end{remark}

\begin{remark}
This statement should be compared with the results of Esnault-Oguiso (\cite{EO} Theorems 5.1, 6.4 and 7.5), who constructed automorphisms which do not lift to characteristic 0. 
\end{remark}

\begin{proof}[Proof of Theorem \ref{liftingauto}] 
Let $\vp: X \ra Y$ be an isomorphism of ordinary K3 surfaces. Consider the graph of this isomorphism as a coherent sheaf (or even as a perfect complex) on the product $X \times Y$, then from Theorem \ref{deformationcomparison} the deformation of isomorphism as a morphism and as a sheaf are equivalent so we use Theorem \ref{liftingkernels} to construct a lifting of the isomorphism for the canonical lift $X_{can}$ of $X$. As isomorphisms preserve the ample cone, the induced Fourier-Mukai transform satisfies the assumptions of Theorem \ref{liftingkernels}. Note that the Lieblich-Olsson lifting of perfect complexes allows us to be only able to choose the lifting of $X$ and then it constructs a unique lifting $Y'$ of $Y$ to which the perfect complex lifts. So, now the only remaining statement to show is that $Y'$ is the canonical lift of $Y$. This follows from the criteria of canonical lift Theorem \ref{lenny} and the observation that the isomorphism between $\vp_{\bar{K}}: X_{can, \bar{K}} \ra Y'_{\bar{K}}$ induces an isomorphism of Galois module on the second $p$-adic \'etale cohomology. This isomorphism of Galois modules provides us with the required decomposition of $H^2_{et}(Y'_{\bar{K}}, \Z_p)$, which shows that $Y'$ is the canonical lift of $Y$. 
\end{proof}

\begin{remark}
This gives a fixed point of the $\delta$ functor constructed by \cite{LO} (see Theorem \ref{Infinitesimallifting}).
\end{remark}

\begin{corollary}
Every isomorphism of ordinary K3 surfaces over an algebraically closed field of characteristic $p$ preserves the Hodge filtration induced by the canonical lift. In particular, the automorphisms as well.
\end{corollary}

\begin{proof}
This follows from \ref{liftingauto} and  \cite{EO} Remark 6.5. 
\end{proof}

\begin{theorem} \label{modulicanonicalliftcommute}
Let $X$ be an ordinary K3 surface, then the canonical lift of the moduli space of stable sheaves with a fixed Mukai vector is the moduli space of stable sheaves with the same Mukai vector on the canonical lift:
\begin{equation} 
(M_X(v))_{can} \cong M_{X_{can}}(v). 
\end{equation}
\end{theorem}

\begin{proof}
We use the criteria for canonical lift Theorem \ref{lenny} to show that $M_{X_{can}}(v)$  is indeed the canonical lift of $M_X(v)$. To use the criteria, we note that 
\begin{equation*}
\begin{split}
H^2_{\acute{e}t}(M_{X_{can}}(v)_{\bar{K}}, \Z_p) &= v^{\perp}/v\Z_p \\
									&\subset H^0_{\acute{e}t}(X_{can, \bar{K}}, \Z_p) \oplus H^2_{\acute{e}t}(X_{can, \bar{K}}, \Z_p) \oplus H^4_{\acute{e}t}(X_{can, \bar{K}}, \Z_p),
\end{split}
\end{equation*}
where the orthogonal complement is taken with respect to the extended pairing on the \'etale Mukai lattice. As $X_{can}$ is the canonical lift of X, we have the following decomposition of
$$
H^2_{\acute{e}t}(X_{can, \bar{K}}, \Z_p) = M^0_X \oplus M^1_X(-1) \oplus M_X^2(-2)
$$
as Galois modules.  We define the decomposition of $H^2_{\acute{e}t}(M_{X_{can}}(v)_{\bar{K}}, \Z_p) = M^0 \oplus M^1(-1) \oplus M^2(-2)$ as Galois modules, where 
\begin{equation}
\begin{split}
M^0 &=  M^0_X \\
M^2 &= M^2_X \\
M^1 &= H^0_{\acute{e}t}(X_{can, \bar{K}}, \Z_p) \oplus H^4_{\acute{e}t}(X_{can, \bar{K}}, \Z_p) \oplus  (v^{\perp}/v\Z_p \cap M^1_X) .
\end{split}
\end{equation}
The last relation above holds using Proposition \ref{etaleversion} and the fact that $H^0_{\acute{e}t}(X_{can, \bar{K}}, \Z_p)$ and $H^4_{\acute{e}t}(X_{can, \bar{K}}, \Z_p)$ are orthogonal to $M_X^1$. 
\end{proof}

\begin{theorem} \label{ordinarycount}
If $X$ is an ordinary $K3$ surface over an algebraically closed field of char $p$, then the number of FM partners of $X$ are the same as the number of Fourier-Mukai partners of the geometric generic fiber of the canonical lift of $X$ over $W$.
\end{theorem}

\begin{proof}
From the discussion in the Chapter \ref{CountingFMP}  Section \ref{surjectivity}, we see that all that is left to show is the injectivity of the specialization map on the set of Fourier-Mukai partners. That is, we need to show that if $M_X(v)$ is isomorphic to $X$, then the lifts of both of them are also isomorphic $X_{can} \cong M_{X_{can}}(v)$. This follows from the definition of canonical lifts and Theorem \ref{modulicanonicalliftcommute} that $M_{X_{can}}(v)$ is the canonical lift of $M_X(v)$. 
\end{proof}

\begin{corollary}
Let $X$ be an ordinary K3 surface over $k$, then the derived autoequivalences satisfying the assumptions of Theorem \ref{Infinitesimallifting} lift uniquely to a derived autoequivalence of $X_{can}$. 
\end{corollary}

\begin{proof}
The argument is going to be similar to the one used to show that every automorphism lifts, but now we will use the proof of Theorem \ref{modulicanonicalliftcommute}. 
Let $\mathcal{P} \in D^b(X \times X) $ induce a derived autoequivalence on $X$, then, using Theorem \ref{Infinitesimallifting}, there exists an $X' / W$ such that we can lift $\mathcal{P}$ to a kernel $\mathcal{P}_W \in D^b(X_{can} \times X')$. Now we need to show that $X'$ is just $X_{can}$. Note that $(\mathcal{P}_W)_{\bar{K}}$ gives a derived equivalence between $D^b(X_{can, \bar{K}}) \cong D^b(X'_{\bar{K}})$, this implies that $X'$ is isomorphic to some moduli space of stable sheaves with Mukai vector $v$, $M_{X_{can, \bar{K}}}(v)$. Now by functoriality of the moduli spaces, we have $M_{X_{can, \bar{K}}}(v) \cong M_{X_{can}}(v)_{\bar{K}}$ and by Theorem \ref{modulicanonicalliftcommute}, we have $M_{X_{can}}(v)_{\bar{K}} \cong M_X(v)_{can, \bar{K}}$. This implies that we get the required decomposition of the second p-adic integral \'etale cohomology of $X'_{\bar{K}}$, which using Theorem \ref{lenny} gives us the result. 
\end{proof}

\begin{corollary}
Every autoequivalence of an ordinary K3 surface that satisfies the assumptions of Theorem \ref{Infinitesimallifting} preserves the Hodge filtration induced by the canonical lift. 
\end{corollary}

\begin{proof}
Follows from the corollary above and Theorem \ref{Hodgefil}.
\end{proof}

\subsection{The Class Number Formula}

Lastly, we give the corresponding class number formula in characteristic $p$ to corollary \ref{classnocount}.

\begin{theorem} \label{Classno}
Let $X$ be a $K3$ surface of finite height over an algebraically field of positive characteristic (say $q >3$). If the N\'eron-Severi lattice of $X$ has rank 2 and determinant $-p$ ($p$ and $q$ can also be same), then the number of Fourier-Mukai partners of $X$ is $(h(p) +1)/2$. 
\end{theorem}

\begin{proof}
We lift $X$ to characteristic 0 using the Lieblich-Maulik Picard preserving lift and then base changing to the geometric generic fiber to get $X_{\bar{K}}$. Choose an embedding of $\bar{K}$ to $\C$ (complex numbers) and base change to $\C$, to get $X_{\C}$. Now, from Proposition \ref{Prop5.2}, we get that every Fourier-Mukai partner of $X$ lifts to a Fourier-Mukai partner of $X_{\C}$. So, we just need to show that if any Fourier-Mukai partner, say $Y_{\C}$, of $X_{\C}$ reduces $\mod q$ to an isomorphic $K3$ surface, say $Y$, to $X$, then it is isomorphic to $X_{\C}$. This follows from noting that if $Y_{\C}$ becomes isomorphic $ \mod q$, then the Picard lattices of $X_{\C}$ and $Y_{\C}$ are isomorphic. The number of Fourier Mukai partners of $X_{\C}$ with isomorphic Picard lattices is given by the order of the quotient of the orthogonal group of discriminant group of $NS(X_{\C})$ by the Hodge isometries of the transcendental lattice (cf. Theorem \ref{CountingformulaChar0}), but in this case the discriminant group of $NS (X_{\C}) = \Z/p$ so the orthogonal group is just ${\pm id}$ and there is always $\pm id$ in the hodge isometries, so we get the quotient to be a group of order $1$.  Thus the result.
\end{proof}

\begin{remark}
Note that the Picard lattice $Pic(X_K)$ and $Pic(X_{\bar{K}})$ are indeed isomorphic as after reduction we are over an algebraically closed field and the line bundles lift uniquely as $Pic_X^0$ is trivial for a K3 surface. 
\end{remark}


\section{Appendix: F-crystal on Crystalline Cohomology} \label{F-crystalchap}

In this appendix, we analyze the possibility of having a ``naive" F-crystal structure on the Mukai isocrystal of a K3 surface. We begin by recalling a few results about crystalline cohomology and the action of Frobenius on it, for details we refer to \cite{SP} Tag 07GI and Tag 07N0,  \cite{Berth},  \cite{BO}, \cite{Liedtke} Section 1.5. 

Let $X$ be a smooth and proper variety over a perfect field $k$ of positive characteristic $p$. Let $W(k)$ (resp.\ $W_m(k)$) be the associated ring of (resp. truncated) Witt vectors with the field of fraction $K$. Let us denote by $Frob_k: k \ra k $; $x \mapsto x^p$, the Frobenius morphism of $k$, which induces a ring homomorphism $Frob_W: W(k) \ra W(k)$, by functoriality, and there exists an additive map $V: W(k) \ra W(k)$ such that $p = V \circ Frob_W = Frob_W \circ V$. Thus, $Frob_W$ is injective.  For any $m > 0$, we have cohomology groups  $H^*_{crys}(X/W_m(k))$. These are finitely generated $W_m(k)$-modules. Taking the inverse limit of these groups gives us the crystalline cohomology:
$$
H^n_{crys}(X/W(k)) : = \varprojlim H^n_{crys}(X/W_m(k)).
$$
It has the following properties as a Weil cohomology theory: 
\begin{enumerate}
\item $H^n_{crys}(X/W(k))$ is a contravariant functor in $X$ and the groups are finitely generated as $W(k)$-modules. Moreover, $H^n_{crys}(X/W(k))$ is $0$ if $n < 0$ or $n > 2dim(X)$.
\item Poincar\'e Duality: The cup-product induces a perfect pairing:
\begin{equation} \label{Poincareduality}
\frac{H^n_{crys}(X/W(k))}{torsion} \times \frac{H^{2dim(X)-n}_{crys}(X/W(k))}{torsion} \ra H^{2dim(X)}_{crys}(X/W(k)) \cong W(k). 
\end{equation}
\item $H^n_{crys}(X/W(k))$ defines an integral structure on $H^n_{crys}(X/W(k)) \otimes_{W(k)} K$.
\item If there exists a proper lift of $X$ to $W(k)$, that is, a smooth and proper scheme $X_W \ra \Spec(W(k))$ such that its special fiber is isomorphic to $X$. Then we have, for each $n$,
$$
H^n_{DR}(X_W/W(k)) \cong H^n_{crys}(X/W(k)).
$$
\item Consider the commutative square given by absolute Frobenius:
\[
\xymatrix{
X \ar[d] \ar[r]^{F} & X \ar[d]\\
k \ar[r]^{Frob_k} & k.}
\] 
This, by the functoriality of the crystalline cohomology, gives us a $Frob_W$-linear endomorphism on $H^i(X/W)$ of $W(k)$-modules, denoted by $F^*$. Moreover, $F^*$ is injective modulo the torsion, i.e., 
$$
F^*: H^i(X/W) / torsion \ra H^i(X/W) / torsion
$$
is injective.
\end{enumerate}

\begin{theorem}[Crystalline Riemann-Roch]
Let $X$ and $Y$ be smooth varieties over $k$, a field of characteristic $p$, and $f: X \rightarrow Y$ be a proper map. Then the following diagram commutes:
\[
\xymatrix{
K_0(X) \ar[d]^{ch( \ ).td_{X}} \ar[r]^{f_*} &K_0(Y)\ar[d]^{ch( \ ).td_{Y}}\\
\oplus_i H^2i_{crys}(X/K) \ar[r]^{f_*} & \oplus_i H^2i_{crys}(Y/K),}
\]
i.e., $ch(f_*\alpha).td_{Y} = f_*(ch(\alpha).td_{X}) \in \oplus_i H^i_{crys}(Y/K)$ for all $\alpha \in K_0(X)$, where $K_0(X)$ is the Grothendieck group of coherent sheaves on $X$.
\end{theorem}

\begin{remark}
The map $f_*$ does not preserve the cohomological grading but does preserve the homological grading, i.e., if the dimensions of $X$ and $Y$ are $n$ and $m$ respectively, then we have the following commutative square: 
\[
\xymatrix{
K_0(X) \ar[d]^{ch( \ ).td_{X}} \ar[r]^{f_*} &K_0(Y)\ar[d]^{ch( \ ).td_{Y}}\\
\oplus_i H^{2i}_{crys}(X/K) \ar[r]^{f_*} & \oplus_i H^{2i +(n-m)}_{crys}(Y/K),}
\] 
and here the grading is respected. If $X$ and $Y$ are K3 surfaces, then $ n = m =2$ and we do not have to worry about this remark, as then the usual cohomological grading is preserved.
\end{remark}

Next we state a few main results about the compatibility of the Frobenius action with the various relations : 

\begin{proposition}[K\"unneth Formula for the crystalline cohomology, \cite{Berth} Chapitre 5, Th\'eor\`eme 4.2.1 and \cite{IllusieCC} Section 3.3]
Let $X, Y$ be proper and smooth varieties over $k$. Then there is a canonical isomorphism in $D(W)$, the derived category of $W$ modules, given as follows:
\begin{equation*}
\mathbb{R}\Gamma(X/W) \otimes_W^{\mathbb{L}} \mathbb{R}\Gamma(Y/W) \cong \mathbb{R}\Gamma(X \times_k Y/W),
\end{equation*}
yielding exact sequences
\begin{eqnarray*}
0 \rightarrow \oplus_{p+q = n} (H^p(X/W) \otimes H^q(Y/W)) \rightarrow H^n(X \times Y/ W) \rightarrow \\
 \rightarrow \oplus_{p + q = n+1} Tor_1^W(H^p(X/W), H^q(Y/W)) \rightarrow 0.
\end{eqnarray*}
\end{proposition}

\begin{remark}
Note that in the case of K3 surfaces the torsion is zero, so we have the following isomorphism:
\begin{equation*}
\oplus_{p+q = n} (H^p(X/W) \otimes H^q(Y/W)) \xrightarrow{\sim} H^n(X \times Y/ W).\\ 
\end{equation*}
\end{remark}

The action of Frobenius gives the following map:
\[
\xymatrix{
F^*H^n(X \times Y/ W) \ar[d]^{=} \ar@{^{(}->}[r] & H^n(X \times Y/W) \ar[d]^{=}\\
\oplus_{p+q = n} (F^*H^p(X/W) \otimes F^*H^q(Y/W)) \ar@{^{(}->}[r] & \oplus_{p+q = n} (H^p(X/W) \otimes H^q(Y/W)).}
\] 

\begin{proposition} \label{Frob-Kun}
The K\"unneth formula is compatible with the Frobenius action in the following way:\\
Let $\gamma \in H^n(X \times Y/W)$ be written (uniquely) as $\gamma = \sum \alpha_p \otimes \beta_q$, then
\begin{equation*}
F^*\gamma = F^*\alpha_p \otimes F^*\beta_q,
\end{equation*}
where $\alpha_p \in H^p(X/W)$ and $\beta_q \in H^q(Y/W)$.
\end{proposition}

Let $p_X$(resp. $p_Y$) denote the projection $X \times Y \ra X$ (resp. $X \times Y \ra Y$).  

\begin{proposition} 
The Frobenius has the following compatibility with the projection morphism:
\begin{equation*}
p_X^*(F^*(\alpha)) = F^*(p_X^*\alpha).
\end{equation*}
Similarly, for the other projection $p_Y$.  
\end{proposition}

Let the denote the cup-product as follows:
\begin{equation*}
H^i(X/W) \times H^j(X/W) \rightarrow H^{i+j}(X/W)
\end{equation*}
given by
\begin{equation*}
( \alpha , \beta) \mapsto \alpha \cup \beta.
\end{equation*}

\begin{proposition} \label{Frob-Cup}
The Frobenius action is compatible with the cup-product in the following way:
\begin{equation*}
F^*(\alpha \cup \beta) = F^*(\alpha) \cup F^*(\beta).
\end{equation*}
Moreover, the Poincar\'e duality induces a perfect pairing as in relation [\ref{Poincareduality}]
$$
<-,->: \frac{H^n}{torsion} \times \frac{H^{2dim(X)-n}}{torsion} \ra H^{2dim(X)} \cong W(k)
$$
which satisfies the following compatibility with Frobenius:
\begin{equation}
 < F^*(x), F^*(y)> = p^{dim(X)} Frob_W(<x,y>).
\end{equation}
\end{proposition}

Now we define an F-crystal (see Definition \ref{F-crystal}) structure on the Mukai F-isocrystal of crystalline cohomology for a K3 surface. \\

Let $X$ be a K3 surface over an algebraically closed field $k$ of characteristic $p >3$. Let $ch = ch_{cris}: K(X) \rightarrow H^{2*}(X/K)$ be the crystalline Chern character and $ch^i$ the $2i-th$ component of $ch$. Reducing to the case of a line bundle via the splitting principle, we see that the Frobenius $\varphi_X$ acts in the following manner on the Chern character of a line bundle $E$:
\begin{equation*}
\varphi_X(ch^i(E)) = p^i ch^i(E).
\end{equation*}
We normalize the Frobenius action on the F-isocrystal $H^*(X/K)$ using the Tate twist to get the {\it Mukai F-isocrystal} $\oplus_i H^i(X/K)(i-1)$. \\

We make the following observation, which shows that how the Frobenius action works on $H^4_{crys}(X/W)$. Note that for a perfect field $k$ of characteristic $p$, Serre (\cite{Serre}, Thm 8, pg 43) showed that the Witt ring W(k) has $p$ as its uniformizer.  Now for $H^4_{crys}(X/K)(1)$ the action of Frobenius is given by $\varphi_X/p$. But note that $ch^2(E) = 1/2(c_1^2(E)-2c_2(E)) $, for $E \in K(X)$, where $c_i(E)$ are the Chern classes of $E$, and as the intersection paring is even for a K3 surface, this is integral, i.e., $ch^2(E) \in H^4(X/W)$. This along with the fact that $rank_W(H^4(X/W)) = 1$  implies that $ch^2(E) = u{p}^n[1] $, where $u \in W^{\times}, p$ is the characteristic of $k$ and $[1]$ is the generator of $H^4(X/W)$ as a $W-$module. Hence, we have 
\begin{eqnarray*}
\varphi_X(ch^2(E))  = \varphi_X (up^n[1])  & = &  \sigma(u p^n) \varphi_X([1] ) \ \text{(via semi-linearity)} \\
								& = & \sigma(u) p^n 	 \varphi_X([1] ) \ \text{(as $\sigma$ is a ring map)}\\					     
								&  = & p^2 \cdot ch^2(E) = p^2 u p^n [1] .
\end{eqnarray*}
 This gives us that
\begin{equation*}
\varphi_X([1]) = u (\sigma(u))^{-1} p^2 [1],
\end{equation*}
where $u (\sigma(u))^{-1} \in W^{\times}$ as $\sigma$  is a ring map. Therefore, we have the Frobenius action on $H^4(X/W) \otimes K(1)$ given by $\varphi'_X([1]) = u (\sigma(u))^{-1} p [1]$. Thus, it indeed has a F-crystal inducing this F-isocrystal given by $(H^4(X/W), \varphi'_X)$. We remark that we are implicitly using the fact that $A \otimes_K K \cong A$, for any $K$-module $A$.

Note that the Mukai vector of a sheaf $P$ in $D^b(X)$ for a K3 surface X is by definition the class 
\begin{equation*}
v(P) = ch(P) \sqrt{td(X)} = (v_0(P), v_1(P), v_2(P)) \in \mathcal{H}^*_{crys}(X/W).
\end{equation*}
Indeed, we have $c_1(X) = 0$ and $2 = \chi(X, \mathcal{O}_X) = td_{2,X}$, which gives us that the Todd genus $td_{X} = (1,0,2) $ and thus $\sqrt{td_{X}} = (1,0,1)$. This then implies that
\begin{equation*}
v(P) = (rk(P), c_1(P), rk(P) + c_1^2(P)/2 - c_2(P)).
\end{equation*}
Note that the intersection pairing on $H^2_{crys}(X/W)$ is even, which gives us the above conclusion as $c_i(P) \in H^{2i}_{crys}(X/W)$ (see \cite{BI}).

\begin{lemma} \label{integral}
The Mukai vector of any object $P \in D^b(X \times Y)$ is a F-crystal cohomology class.
\end{lemma}
\begin{proof}(cf. \cite{Mu})
Note that from the definition of the F-crystal structure we just need to show that $ch(P) \in H^*_{crys}(X \times Y /W)$ as the square root of the Todd genus for a K3 surface is computed as follows:
\begin{equation*}
\sqrt{td_{X \times Y}} = p_1^*\sqrt{td_{X}}p_2^*\sqrt{td_{Y}} = p_1^*(1,0,1).p_2^*(1,0,1). 
\end{equation*}
We write the exponential chern character as follows:
\begin{equation*}
ch(P) = (rk(P), c_1(P), 1/2(c_1^2(P)-2c_2(P)), ch_3(P), ch_4(P))
\end{equation*}
where
\begin{equation*}
ch^3(P)= 1/6(c_1^3(P) - 3c_1c_2 + 3c_3(P))
\end{equation*}
and
\begin{equation*}
ch^4(P) = 1/24(c_1^4 - 4c_1^2c_2 + 4c_1c_3 + 2c_2^2 - 4 c_4).
\end{equation*}
Note that if $char(k) \neq 2, 3$, then $2, 3 $ are invertible in $W(k)$, so $ch(P) \in H^*_{crys}(X \times Y /W)$ as again we know $c_i(P) \in H^{2i}_{crys}(X \times Y /W)$ . 
\end{proof}

\begin{remark}
Thus, it makes sense to talk about the descent of a Fourier-Mukai transform to the F-crystal level but note that the new Frobenius structure on $H^4(X/W)(1)$ fails to be compatible with the intersection pairing as defined in Theorem \ref{Frob-Cup}. This causes the failure of existence of an F-crystal structure on the Mukai-isocrystal and also the failure to have a cohomological criteria of derived equivalences of K3 surfaces with crystalline cohomology.  
\end{remark}



\begin{thebibliography}{xx}
\bibitem{SP} de Jong A. et al., {\it Stacks Project}, \url{http://stacks.math.columbia.edu}, 2018. 
\bibitem{ArtinRT} Artin M., {\it Versal deformation and algebaric stacks}, Inventiones math. 27, 165-189, 1974.
\bibitem{AM} Artin M., Mazur B., {\it Formals groups arising from algebraic varieties}, Ann. Sc. \'Ec. Norm. Sup. 4e s\'erie, t. 10, 87-132, 1977.
\bibitem{Barth} Barth B., Hulek K., A. M. Peters C., Van De Ven A., {\it Compact Complex Surfaces}, A Series of Modern Surveys in Mathematics, 4, Springer, 2003. 
\bibitem{BBR} Bartocci C., Bruzzo U., Ruip\'erez D.H., {\it Fourier-Mukai and Nahm Transforms in Geometry and Mathematical Physics}, Progress in  Mathematics, 276, Birkh\"auser, 2009.
\bibitem{BB} Bayer A., Bridgeland T., {\it Derived automorphism groups of $K3$ surfaces of Picard rank 1}, Duke Math. J., 166, Number 1, 75-124, 2017.
\bibitem{BBD} Beilinson, A. A., Bernstein J., Deligne, P., {\it Faisceaux pervers}, Ast\'erisque, Soc. Math. France, 100, 5-171, 1982.
\bibitem{BZ-F-N} Ben-Zvi D., Francis J., Nadler D., {\it Integral Transforms and Drinfeld Centers in Derived Algebraic Geometry}, Journal of the American Mathematical Society, 23(4), 909-966, 2010.
\bibitem{Berth} Berthelot P., {\it Cohomologie cristalline des schémas de caractéristique $p > 0$}, Lecture Notes in Math. 407, Springer-Verlag, 1974.
\bibitem{BOcrys} Berthelot, P., Ogus, A.,{\it Notes on crystalline cohomology}, Annals of Math. Lecture Notes, Princeton University Press, 1978. 
\bibitem{BO} Berthelot, P., Ogus, A., {\it F-isocrystals and de Rham cohomology. I.} Inventiones mathematicae 72,  159-200, 1983.
\bibitem{BI} Berthelot, P., Illusie, L., {\it Classes de Chern en cohomologie cristalline}, C.R. Acad Sci. Series A, 270, 1695-1697, 1750-1752, 1970.
\bibitem{BOderived} Bondal A., Orlov D., {\it Reconstruction of a variety from the derived category and groups of autoequivalences}, Comp. Math. 125, 327-344, 2001.  
\bibitem{Bri1} Bridgeland T., {\it Stability conditions on triangulated categories}, Annals of Mathematics, Second Series, 166, No. 2, 317-345, 2007.
\bibitem{Bri2} Bridgeland T., {\it Stability conditions on $K3$ surfaces}, Duke Math. J., 141, Number 2, 241-291, 2008.
\bibitem{Bri3} Bridgeland T., {\it Space of stability conditions}, Algebraic geometry-Seattle 2005. Part 1, 1-21, Proc. Sympos. Pure Math., 80,  Amer. Math. Soc., Providence, RI, 2009.
\bibitem{Deligne} P. Deligne, {\it Relevement des surfaces K3 en characteristique nulle}, Lecture notes in Math 868, Algebraic surfaces, 58-79, Springer, 1981.
\bibitem{Deligne-Illusie} Deligne P., Illusie L., {\it Cristaux ordinaries et coordn\'ees canoniques}, In surfaces Algebrique, Seminar Orsay, 1976-78. Lecture Notes in Math., 868, Springer-Verlag, 1981. 
\bibitem{EO} Esnault H., Oguiso K., {\it Non-liftablility of automorphism groups of K3 surface in positive characteristic},  Math. Ann. 363, 1187-1206, 2015.
\bibitem{GM} Gelfand S., Manin Y. {\it Methods of Homological Algebra}, Springer Monographs in Mathematics, 1997. 
\bibitem{Grothendieck} Grothedieck A., {\it \'El\'ements de g\'eom\'etrie alg\'ebrique}, III, \'Etude cohomologiue des faisceaux coh\'erents. I. Inst. Hautes \'Etudes Sci. Publ. Math., 11:167, 1961.
\bibitem{SGA4} Grothendieck A., {\it Th\'eorie des topos et cohomologie \'etale des sch\'emas}, 4, Tome III, Lecture Notes in Mathematics, 305, 1972.
\bibitem{tohoku} Grothendieck A., {\it Sur quelques points d'algèbre homologique}, I. Tohoku Math. J. (2) 9, no. 2, 119-221, 1957.
\bibitem{gomez} G\'omez T., {\it Algebraic stacks}, Proc. Indian Acad. Sci. Math. Sci., 111 (1), 1-31, 2001.
\bibitem{HartDR} Hartshorne R., {\it On the de Rham cohomolgy of algebraic varities}, Publications Mathematiques De L\'IHES, 45, 5-99, 1975.
\bibitem{HartshorneAG} Hartshorne R., {\it Algebraic geometry},  Graduate Text in Mathematics, 52, Springer, 1977.
\bibitem{HartshorneDT} Hartshorne R., {\it Deformation Theory}, Graduate Text in Mathematics, 257, Springer, 2010.
\bibitem{Hosono} Hosono S., B.H. Lian, K. Oguiso, S-T.Yau, {\it Autoequivalences of derived category of a K3 surface and monodromy transformations}, J. Alg Geom., 13, 513-545, 2004.
\bibitem{HLOY} Hosono, S. and Lian, B. and Oguiso, K. and Yau, S.-T., {\it Fourier-Mukai Number of a K3 Surface}, CRM Proc. Lecture Notes, 38, 2004.
\bibitem{HuyFM} Huybrechts D., {\it Fourier Mukai Transforms in Algebraic Geometry}, Oxford Science Publication, 2006.
\bibitem{HL}  Huybrechts D., Lehn M., {\it The geometry of moduli spaces of sheaves}, Second edition, Cambridge Mathematical Library, Cambridge University Press, 2010. 
\bibitem{Huy} Huybrechts, D., Marci, E., Stellari, P., {\it Derived Equivalences of K3 Surfaces and Orientation}, Duke Math J. 149, 461-507, 2009.
\bibitem{Huy2} Huybrechts D., Marci E., Stellari P., {\it Stability conditions for generic K3 categories}, 32 pages. Comp. math. 144, 134-162, 2008.
\bibitem{HuySC} Huybrechts D., {\it Introduction to stability conditions}, in 'Moduli Spaces' Cambridge University Press edited by Brambila Paz, Newstead, Thomas and Garcia-Prada Lectures at the Newton Institute January 2011, 47 pages.
\bibitem{HuyRT} Huybrechts D., Richard T., {\it Deformation-Obstruction theory for complexes via Atiyah and Kodaira-Spencer Classes}, Math. Ann. 346, 545-569, 2013.
\bibitem{HuyLect} Huybrechts D., {\it Lectures on K3 surfaces}, Cambridge University Press, 2016.
\bibitem{IllusieCC} Illusie L.,  {\it Report on Crystalline cohomology}. In Algebraic geometry (Proc. Sympos. Pure Math., Vol. 29, Humboldt State Univ., Arcata, Calif., 1974), American Mathematical Society, Providence, R.I., 459–478, 1975.
\bibitem{Katz} Katz N., {\it Slope Filtration of F-crystals}, Ast\'erisque 63, 113-164, 1979.
\bibitem{KS} Kashiwara M., Schapira P., {\it Sheaves on manifolds}, Grundlehren 292, Springer, 1990. 
\bibitem{Knutson} Knutson D., {\it Algebraic spaces}, Lecture Notes in Mathematics 203, 1971. 
\bibitem{Ko} Kontsevich M., {\it Homological algebra of mirror symmetry}, Proceedings of the International Congress of Mathematicians (Z\"urich 1994), Birkh\"auser, 120-139, 1995.
\bibitem{Langer} Langer A., {\it Semistable Sheaves in positive characteristic}, Annals of Math., 159, 251-276, 2004. 
\bibitem{L} Lieblich M., {\it Moduli of complexes on a proper morphism}, J. Algebraic Geometry, 15, pp.175-206, 2006. 
\bibitem{L2} Lieblich M., {\it Moduli of Twisted sheaves}, Duke Math J., 138, 23-118, 2007. 
\bibitem{LO} Lieblich M., Olsson M., {\it Fourier Mukai partners of K3 surfaces in positive characteristic}, Annales Scientifiques de L’ENS, 48, fascicule 5, 1001-1033, 2015.
\bibitem{LO2} Lieblich M., Olsson M., {\it A Stronger Derived Torelli Theorem for K3 surfaces}, In: Bogomolov F., Hassett B., Tschinkel Y. (eds) Geometry Over Nonclosed Fields. Simons Symposia. Springer, Cham, 2017.
\bibitem{LM} Lieblich M., Maulik D., {\it A note on the cone conjecture for K3 surfaces in positive characteristic}, preprint.
\bibitem{LiedtkeMats} Liedtke, C., Matsumoto, Y., {\it Good reduction of K3 surfaces}, Compos. Math. 154, 1-35, 2018.
\bibitem{Liedtke} Liedtke C., {\it Lectures on Supersingular K3 surfaces and the Crystalline Torelli Theorem}, In: K3 Surfaces and their Moduli, Progress in Mathematics 315, Birkh\"auser, 171-325, 2016.
\bibitem{LurieFM} Lurie J., {\it Derived Algebraic Geometry X: Formal Moduli Problems}, http://www.math.harvard.edu/~lurie/papers/DAG-X.pdf, 2011. 
\bibitem{Manin} Manin Y., {\it Theory of commutative formal groups over fields of finite characteristic}, Uspehi Mat. Nauk SSSR, 18(6 (114)), 3-90, 1963.
\bibitem{Matsumoto} Matsumoto Y., {\it Good reduction criteria for K3 surfaces}, Mathematische  Zeitschrift,  2015.
\bibitem{MukaiAV} Mukai S., {\it Duality between $D(X)$ and $D(\hat{X})$ with its applications to Picard sheaves}, Nagoya Math. J. 81, 153-175, 1981. 
\bibitem{Mu} Mukai S., {\it On the Moduli space of bundles on K3 surfaces I.}, In: Vector bundles on algebraic varieties, Bombay, 1984.
\bibitem{MM} Matsusaka T., Mumford D., {\it Two theorems on deformations of polarized varities}, Amer. J. Math., 86, pp. 668-684, 1964.
\bibitem{NNQuot} Nitsure N., {\it Construction of Hilbert and Quot schemes}, Fundamental algebraic geometry: Grothendieck’s FGA explained, Mathematical Surveys and Monographs 123, American Mathematical Society, 105–137, 2005.
\bibitem{Nygaardtorelli} Nygaard N., {\it The Torelli theorem for ordinary K3 surfaces over finite fields}, Arithmetic and geometry I, Progress in Mathematics 35, 267-276, Birkh\"auser, 1983.
\bibitem{Nygaardtate} Nygaard N.,  {\it Tate conjecture for ordinary K3 surfaces over finite fields}, Inv. Math. 74, 213-237, 1983. 
\bibitem{OlssonAS} Olsson M., {\it Algebraic Spaces and Stacks}, Colloquium Publications, 62, American Mathematical Society, 2016.
\bibitem{Orlov1} Orlov D., {\it On equivalences of derived categories and K3 surfaces}. J. Math Sci. (New York), 84, 1361-1381, 1997.
\bibitem{Orlov2} Orlov D., {\it Derived categories and coherent sheaves and equivalences between them}, Russian Math. Surveys, 58, 511-591, 2003.
\bibitem{Ogus2} Ogus A., {\it Supersingular K3 crystals}, Journ\'ees de G\'eom\'etrie Alg\'ebraique de Rennes Vol. II, Ast\'erisque 64, 3-86, 1979.
\bibitem{Ogus} Ogus A., {\it A crystalline torelli theorem of supersingular K3 surfaces}, Arithmetic and Geometry II, Progress in Mathematics 36, 361-394, Birkh\"auser, 1983.
\bibitem{Ploog} Ploog D., {\it Group of Autoequivalences of derived categories of smooth projective varieties}, PhD thesis, Freie Universit\"at Berlin, 2005.
\bibitem{RS76} Rudakov A. N., Shaferevich I. R., {\it Inseparable morphisms of algebraic surfaces}, Izv. Akad. Nauk SSSR 40, 1269-1307 (1976). 
\bibitem{Serre} Serre J.P., {\it A course in Arithmetic}, Graduate Text in Mathematics, 7, Springer, 1973. 
\bibitem{SerDef} Sernesi E., {\it Deformation of Algebraic varieties}, Grundlehren der mathematischen Wissenschaften, 334, Springer-Verlag Berlin Heidelberg, 2006.
\bibitem{SrivasPhD} Srivastava T.K., {\it On Derived Equivalences of K3 surfaces in positive characteristic}, PhD Thesis, Freie Universit\"at Berlin, 2018. 
\bibitem{Taelman} Taelman L., {\it Ordinary K3 surfaces over finite fields}, preprint: https://arxiv.org/abs/1711.09225, 2017.
\bibitem{MatthewWard} Ward M., {\it Arithmetic Properties of the Derived Category for Calabi-Yau Varieties}, PhD thesis, University of Washington, 2014.

\end{thebibliography}
\end{document}